\documentclass{amsart}

\usepackage{amssymb,amsmath, amscd, mathrsfs, yfonts, bbm,
  graphics, dsfont, footmisc, tikz, amsthm}

\usepackage{graphicx, caption}
\usepackage{changepage}
\input xy
\xyoption{all}

\newtheorem{claim}{}[section]
\newtheorem{defn}[claim]{Definition}
\newtheorem{thm}[claim]{Theorem}
\newtheorem{lemma}[claim]{Lemma}
\newtheorem{remark}[claim]{Remark}

\newtheorem{cor}[claim]{Corollary}

\newtheorem*{thmnn}{Theorem}

\newcommand{\cJ}{{J}(m)}

\newcommand{\Ltr}{{\reflectbox{$\Gamma$}}}

\newcommand{\ltr}{{\reflectbox{\tiny$\Gamma$}}}

\newcommand{\oj}{\vec{j}}

\newcommand{\ou}{\vec{u}}
\newcommand{\os}{\vec{\sigma}}

\newcommand{\ab}{\allowbreak}

\newcommand{\itd}{\textit{\noindent\romannumeral\day.%
\romannumeral\month.\romannumeral\year}}

\newcommand{\Tr}{\mathrm{Tr}}
\newcommand{\tr}{\mathrm{tr}}

\newcommand{\cA}{\mathcal{A}}

\newcommand{\cI}{\mathcal{I}( m)}
\newcommand{\I}{\mathcal{I}}

\newcommand{\cP}{\mathcal{P}}

\title[Partial Transpose and Asymptotic Freeness] {The
  Partial Transpose and Asymptotic Free Independence for
  Wishart Random Matrices:\\ Part II}

\author[mingo and popa]{James A. Mingo$^{(*)}$ and Mihai Popa$^{(**)} $ }

\address{Department of Mathematics and Statistics, Queen's
  University, Jeffery Hall, King\-ston, Ontario, K7L 3N6,
  Canada} \email{mingo@mast.queensu.ca}

\thanks{$^{(*)}$Research supported by a Discovery Grant from the
Natural Sciences and Engineering Research Council of Canada. \hfill \itd}

\thanks{$^{ (* *) }$ Research supported by the Simons
	Foundation grant No. 360242.}

\address{Department of Mathematics, University of Texas at
  San Antonio, One UTSA Circle San Antonio, Texas 78249,
  USA, and}

\address{``Simon Stoilow'' Institute of Mathematics of the
  Romanian Academy, P.O. Box 1-764, 014700 Bucharest,
  Romania}

\email{mihai.popa@utsa.edu}

\begin{document}

\begin{abstract}
Using new combinatorial techniques, we significantly improve
the previous results on asymptotic distributions and asymptotic free independence relations of partial transposes of Wishart random matrices. In particular, we give a necessary
and sufficient condition for the asymptotic free
independence of partial transposes of Wishart matrices with
difference block sizes.
\end{abstract}

\maketitle

\section{Introduction and Statement of Results}

Since their introduction in the first half of the 19th century (probably in \cite{wishart}), Wishart matrices have been heavily used and studied in high dimensional statistics, in connection to problems arising from multivariate analysis of variance (see \cite{muirhead}, \cite{edelman}). In recent years, Wishart random matrices and their partial transposes appeared in the literature on quantum information theory (see \cite{aubrun}, \cite{aubrun2}) in connection with entanglement properties. This motivated the study of the asymptotic behavior of partial transposes, a subject that was not addressed much by the existing work. 

An intuitive definition of partial transposes is as follows.
A
$ bd \times bd $
matrix, 
$ X $
can be seen as 
$ b \times b $
block matrix, each entry being a 
$ d \times d $ 
matrix.
The $(b, d)$-partial transpose of 
$ X $,
here denoted by
$ X^{\Gamma(b, d)} $,
is obtained by transposing each 
$ d \times d $ 
block, without modifying the position of the blocks.
For example 
\[\begin{pmatrix} A_{11} & A_{12} \\ A_{21} & A_{22} \end{pmatrix}^{\Gamma(2,d)}
\kern-1em =
\begin{pmatrix} A_{11}^\mathrm{T} & A_{12}^\mathrm{T} \\ 
                A_{21}^\mathrm{T} & A_{22}^\mathrm{T} \end{pmatrix}.\]

In \cite{aubrun} it is shown that, for 
$ W $ 
a Wishart random matrix, the asymptotic distribution of 
$ W^{ \Gamma(N, N)} $
is shifted semicircular. When 
$ b $ 
is fixed and 
$ d \rightarrow \infty $, the asymptotic distribution of 
$ W^{ \Gamma(b, d)} $ was computed in 
\cite{banica-nechita} (see also \cite{arizmendi} and \cite{fs}). 
In this case the limit distribution is the rescaled free difference of two Marchenko-Pastur distributions. 
In part I of this series, \cite{mingo-popa-wishart}, we computed the asymptotic distribution of 
$ W^{ \Gamma(b, d)} $
for general 
$ b $ and $ d $
and described the asymptotic relations between
$ W $, $ W^{ \Gamma(b, d)} $
and their transposes. In particular, it is shown that
$ W $ 
and 
$ W^{ \Gamma(b, d)} $
are asymptotically free if and only if
 $ d \rightarrow \infty $, 
 while
 $ W $ and
 $ W^{\ltr(b, d)}$, 
 the matrix transpose of 
   $ W^{\Gamma(b, d)} $,
 are asymptotically free if and only if 
 $ b \rightarrow \infty $.

%
%
%
%
%
%

The present paper uses new combinatorial techniques, inspired by \cite{popa-guassian}, to significantly improve the results from Part I. Among other (more technical) questions, we address the asymptotic freeness relations between different partial transposes of the same Wishart random matrix. More precisely, the necessary and sufficient condition below is a particular case of the main result (see Theorem \ref{thm:main} and Corollary \ref{cor:4:15}).

\begin{thmnn}
Suppose that 
$( b_N)_N $,
 $ ( d_N ) _N $
 $ (b_N^\prime)_N $
 and
 $ (d_N^\prime)_N $ 
 are four non-decreasing sequences of positive integers such that
 $ b_N \cdot d_N = b_N^\prime \cdot d_N^\prime  = M_N$ 
 for each 
 $ N $,
  and that
 $ \displaystyle \lim_{N \rightarrow \infty}
  b_N \cdot d_N = \infty $. 
  Suppose also that each of the permutations 
  $ \gamma ( b_N, d_N ) $ 
  and 
  $ \gamma (b_N^\prime, d_N^\prime ) $
  is either
   the partial transpose
   $ \Gamma (b_N, d_N) $, 
   respectively
    $ \Gamma (b_N^\prime, d_N^\prime) $
    or its matrix transpose
    $ \Ltr (b_N, d_N) $, 
    respectively
    $ \Ltr( b_N^\prime, d_N^\prime ) $. 
    
    Then
     $ W^{\gamma( b_N, d_N)} $ 
     and
      $ W^{ \gamma(b_N^\prime, d_N^\prime)} $
      are asymptotically free if and only if the entry permutation
      $ \gamma( b_N , d_N)^{ -1} \circ \gamma(b_N^\prime, d_N^\prime ) $ 
      has
      $ o(M_N^2) $
      fixed points.
\end{thmnn}

Besides the Introduction, the paper is organized into 3 more sections. Section 2 presents the notations and several technical results concerning symmetric (i.e. preserving self-adjointness) entry permutations on Wishart matrices. The techniques in this section are inspired from the results in \cite{popa-guassian} on Gaussian random matrices. Section 3 presents a combinatorial inequality satisfied by partial transposes. Section 4 presents the main results, giving necessary and sufficient conditions for the asymptotic freeness of different partial transposes  and left partial transposes of Wishart random matrices. Finally, Section 5 presents a result on the boundedness of covariance of traces of products of various partial transposes of Wishart matrices; using standard procedures from probability, it follows that the results from the previous Section hold true almost surely.

\section{Some results on symmetric entry permutations of
  Wishart Random Matrices } \label{section:vV}

For a positive integer $ N $, we will denote by $ [ N ]$ the
ordered set $ \{ 1, 2, \dots, N \} $ and by $ \mathcal{S}( [
  N ]^2)$ the set of permutations on $ \{ (i, j) : 1 \leq i,
j \leq N \} $, i.e.
\[
\mathcal{S} ([ N ]^2) = \{ \sigma : [ N ] \times [ N ]
\rightarrow [ N ] \times [ N ], \textrm{bijection} \}.
\]
If $ A $ is a $ N \times N $ square matrix and $ \sigma \in
\mathcal{S}([ N ]^2) $, we denote by $ A^\sigma $ the matrix
defined by $ \big[ A^\sigma\big]_{i, j} = \big[ A \big]_{
  \sigma(i, j)} $.  A permutation $ \sigma \in \mathcal{S}([
  N ]^2) $ will be called \emph{symmetric} if it commutes
with the transpose, i.e.
\[ 
\sigma \circ t (a, b) = t \circ \sigma (a, b) 
\]
where $ t (a, b) = (b, a) $.  Note that if $ \sigma $ is
symmetric, then $ \sigma(a, b) = (c, c) $ is equivalent
to $ a = b $.  Moreover, if $ A $ and $ \sigma $ are
symmetric, then so is $ A^\sigma $.

In this paper, we will define a \emph{Wishart} random matrix
with shape parameters $(M, P)$ as follows.  First let $ G =
\big[ g_{i, j} \big]_{(i, j)\in [ M ] \times [ P ]} $ to be
a \emph{Ginibre} rectangular $ M \times P $ random matrix,
i.e. the its entries form a family $ \big\{ g_{i, j}: i \in
[ M ], j \in [ P ] \big\} $ of independent, identically
distributed complex Gaussian random variables with mean $ 0
$ and complex variance $ \displaystyle \frac{1}{ \sqrt M } $
and then we set $W =G G^*$.

In this paper, we shall suppose that
$ \big\{ M_N \big\}_N $ 
and
$ \big\{ P_N \big\}_N $ 
are two strictly increasing sequences of positive integers such that
$ \displaystyle
\lim_{N \rightarrow \infty} \frac{P_N}{M_N} = c $
for some (fixed)
$ c > 0 $ and $W_N$ will denote the Wishart matrix with these shape
parameters.

In the next paragraphs we will introduce some notations
concerning the expression
\[
E \circ \tr \big( W^{ \sigma_{1}} \cdot W^{ \sigma_{2}}
\cdots W^{ \sigma_{m}} \big)
\]
where $ \sigma_{ k} $ is a symmetric permutation from
$\mathcal{S} ([ M ]^2) $ for each $ k = 1, 2, \dots, m $.

First, with the convention
$ i_{  m+1 } = i_1 $,
denote by 
$ \mathcal{I}(m) $ the set 
\begin{align*}
\{ \ou = (i_1, j_1, j_{-1}, i_{-1},
i_2, j_2, \dots, 
i_{m}, j_m,  j_{-m},& i_{-m})  
:\ 
i _{ \pm k} \in [ M ], 
j_{\pm k } \in [ P ] \\
\textrm{ and } & i_{-k}= i_{k +1}
\textrm{ and }
j_k = j_{-k}
\ \textrm{for}\ k \in [ m ] \}.
\end{align*}
With this notation, we have then:
\begin{align}
E \circ  \tr 
\big( 
W^{ \sigma_{1}}  & 
\cdots 
W^{ \sigma_{ m}} 
\big)
=
\sum_{\substack{t_k  \in [ M ] \\ 1\leq k \leq m }} \frac{1}{M} 
E \Big(    
[ W^{\sigma_1}]_{t_1 t_{2}}
[ W^{\sigma_2}]_{t_2 t_{3}}
\cdots [ W^{ \sigma_m}]_{t_m t_{1}}
\Big) \label{formula:1} \\
& = 
\sum_{ \ou \in \mathcal{I}(m)} 
\frac{1}{M}
E  \big( 
g_{ \pi_1 \circ \sigma_1 (i_1, i_{-1}), j_1}
\overline{
g_{ \pi_2 \circ
\sigma_1 (i_1, i_{-1}), j_1}
}
\cdots \nonumber \\
& \hspace{4cm}\cdots
g_{ \pi_1 \circ 
\sigma_m (i_m, i_{-m}), j_m}
\overline{
g_{ \pi_2 \circ 
\sigma_m (i_m, i_{-m}), j_m}
}
\big)\nonumber\\
& = 
\sum_{ \ou \in \mathcal{I}(m)} 
\frac{1}{M}
E  \big( 
g_{ l_1, j_1}
\overline{
g_{ l_{-1}, j_{-1}}
}
\cdots
g_{l_m, j_m}
\overline{ g_{l_{-m}, l_{-m} } } \big)\nonumber  
\end{align}  
where
$ \pi_1 $ and $ \pi_2 $ 
are the canonical projections
(i.e. 
$ \pi_1(i, j) = i $ 
and
$ \pi_2 (i, j) = j $) 
and
\begin{align*}
& l_k =  \pi_1 \circ \sigma_k ( i_k, i_{-k})  \\
& l_{-k} = \pi_2 \circ \sigma_k ( i_k, i_{-k}).  
\end{align*}

Equation (\ref{formula:1}) can be further refined using
Wick's formula (see \cite{jason} or \cite[\S 1.5]{mingo-speicher}). More precisely, denote by
$ \cP_2(2m, 2) $ the set of pair partitions on $ [2 m ] $ such
that $ k + \pi(k) $ is odd for each $ k $. Here we are thinking
of $\pi$ as the permutation of $[2m]$ where each block of the
partition becomes an cycle of the permutation. 
 Wick's formula
gives then
\begin{align*}
E \circ \tr \Big( 
& W^{ \sigma_{1}}
\cdot W^{ \sigma_{ 2}} 
\cdots 
W^{ \sigma_{ m}} 
\Big)
=
\sum_{\pi \in \cP_2(2m, 2)}
\frac{1}{M}
\sum_{ \ou \in \mathcal{I}(m)}
\prod_{ (2t-1, 2s) \in \pi}
E \Big( 
g_{ l_t, j_t }
\overline{   g_{ l_{-s}, j_{-s}}  }     
\Big)
\end{align*} 

Therefore, denoting 
\[
v(\pi, \overrightarrow{\sigma}, \ou )
=  \prod_{ (2t-1, 2s) \in \pi}
E \Big( 
g_{ l_t, j_t
}
\overline{   g_{ l_{-s}, j_{-s}}  }     
\Big)
\]
and 
\begin{align*}
\mathcal{V}(\pi, \overrightarrow{\sigma}) 
& =
\frac{1}{M}
\sum_{ \ou \in \mathcal{I}(m) }
v( \pi, \os, \ou ).
\end{align*}
we have  that
\begin{equation}\label{eq:wick}
E \circ \tr \big(
W^{\sigma_1} \cdot W^{\sigma_2}
\cdots W^{\sigma_m}  \big)
= \sum_{ \pi \in \cP_2(2m, 2)} 
\mathcal{V}(\pi, \os) .
\end{equation}       

Moreover, denoting
\[
\mathcal{A}(\pi, \os) 
= \big\{
\ou\in \mathcal{I}(m)  :\
v( \pi, \os, \ou)
\neq 0 
\big\},
\]  
we have that
\begin{equation}\label{v:card}
\mathcal{V}( \pi, \os) = M^{ - m -1 }  \cdot
\# (\mathcal{A}(\pi, \os)).
\end{equation}

To simplify the writing in the next lemma, we need more notation. First, let 
$ D = \{ d_1, d_2, \dots, d_r \} $
be a subset of 
$ [ 2 m ] $
such that 
$ d_1 < d_2 < \dots < d_r $.
For
$ \overrightarrow{w}
= (w_1, w_2, \dots, w_{4m}),
$
denote
\[
\overrightarrow{w}[D] = (w_{2d_1 -1}, w_{2d_1},
w_{2d_2 -1}, w_{2d_2},
\dots, w_{2d_r-1}, w_{2d_r}).
\] 

Next, for 
$ \ou
=( i_1, j_1, j_{-1}, i_{-1},
\dots, j_{-m}, i_{-m})
\in \mathcal{I}(m) $,
denote
\begin{align}
\os(\ou)  =  & (l_1, j_1, j_{-1}, l_{-1}, 
\dots, l_m, j_m, j_{-m}, l_m ) \label{sigma:u}\\
=& \big(
\pi_1 \circ \sigma_1(i_1, i_{-1}), j_1,
j_{-1}, \pi_2\circ \sigma_1 ( i_1, i_{-1}),
\dots,
j_{-m}, \pi_2 \circ \sigma_m (i_m, i_{-m})
\big)\nonumber,
\end{align}
%
and define
\[
\cA_{\pi, \os }(D) =\big\{
\os( \ou)[ D ] \mid \ 
\ou \in \cA(\pi, \vec\sigma) 
\big\}.
\]

\begin{lemma}\label{lemma:1:1}
Let $ B$ be a subset of $ [ 2m ] $ 
which is closed with respect to $ \pi \in \cP_2(2m, 2)$, 
i.e.
if $ l \in B $, then $ \pi(l) \in B $.
Suppose that
$\big\{ 2k +1, 2k + 2 \big\} \subseteq B $
and let
\begin{align*}
B_1 &
= B \cup \big\{ 2k-1, 2k
\big\}\cup \big\{ \pi(2k-1), \pi(2k)
\big\} \mathrm{\ and\ }\\
B_2 &
= B \cup
\big\{ 2k+3, 2k + 4
\big\}
\cup
\big\{
\pi(2k+3), \pi(2k + 4)
\big\}.
\end{align*}
Then, for $j =1,2$, we have that
\[
\#(\cA_{\pi, \os}(B_j))
\leq 
\#(\cA_{ \pi, \os }(B))
\cdot
(\max\{ M, P\})^{\frac{1}{2} \#\big( B_j \setminus B\big) }.
\]
\end{lemma}
\begin{proof}
We will give the details for $ j =1$; the case $ j =2 $ is similar.

Fix
$ \overrightarrow{\alpha}\in
\cA_{ \pi, \os}(B) $
and suppose that 
$ \ou \in \mathcal{A}( \pi, \os) $ 
and
$ \vec{\beta} \in
\cA_{ \pi, \os }(B_1) $
are such that
$ \os(\ou)[ B ] = \vec{\alpha} $
and 
$ \os(\ou)[B_1] = \vec{\beta} $.

We shall distinguish three cases, if none, both or only one elements of the set
$ \{ 2k-1, 2k \} $
are in 
$ B_1 \setminus B $.

If 
$ \{ 2k-1, 2k \} \subseteq B $, 
then 
$ B_1 = B $ 
and the conclusion follows trivially.

Next, suppose that 
$ \big\{ 2k-1, 2k \big\} \subseteq B_1 \setminus B $. 
Again, we distinguish two cases, depending on the equality between 
$ \pi(2k-1)$ and $ 2k $.

If 
$ \pi (2k-1) = 2k $,
then, by construction, 
$ B_1 =  \big\{ 2k-1, 2k \big\} \cup B $.

Writing
\begin{align*}
& \ou = (i_1, j_1, j_{-1}, i_{-1}, \dots, i_m, j_m, j_{-m}, i_{-m})
\mathrm{\ and\ }\\
& \os(\ou)  =  (l_1, j_1, j_{-1}, l_{-1}, 
\dots, l_m, j_m, j_{-m}, l_{-m})
\end{align*}
we have that the only components of 
$ \os(\ou) $ 
that are components of
$ \vec{\beta} $
but not of 
$ \vec{\alpha} $
are
$ l_k, j_k, j_{-k}, l_{-k} $.

Since 
$ \ou \in \mathcal{A} (\pi, \os ) $,
we have that
$ 
E \big( g_{l_k, j_k} 
\overline{g_{l_{-k}, j_{-k}}} \big) \neq 0.
$ 
Hence 
$ j_k = j_{-k} $
and 
$ l_k = l_{-k} $.
The last equality means that
\[ \pi_1 \circ \sigma_k (i_k, i_{-k})
= \pi_2 \circ \sigma_k ( i_k, i_{-k}), \] 
which, since
$ \sigma_k $
is symmetric, is equivalent to 
$ i_k = i_{-k} $. 

On the other hand, 
$ \{ 2k+1, 2k + 2 \} \in B $, 
so
$l_{k+1}$
and 
$ l_{-(k+1)} $ 
are components of 
$ \vec{\alpha}$.
But 
\[ i_{-k}= i_{k+1} = \pi_1 \circ \sigma_{k+1}^{-1}(l_{k+1}, l_{-(k+1)}),
\]
that is,
$ i_{-k} $
is uniquely determined by
$\vec{\alpha}$.

If follows that all the components of
$ \vec{\beta} $ 
are uniquely determined by
$ \vec{\alpha} $
and by
$ j_k  = j_{-k} $.
So
\[
\#( \cA_{ \pi, \os }(B_1)) \leq P \cdot
\#( \cA_{ \pi,\os }(B))
= P^{\frac{1}{2} \#(B_1 \setminus B) }
\cdot \#( \cA_{ \pi,\os  }(B)).
\]

If
$ \pi(2k-1) \neq 2k $, 
then
$ \pi (2k-1) \not = 2k $ and $\{ 2k -1 , 2k, \pi(2k -1), \pi(2k)\}$
are distinct and they are not elements of 
$ B $. 
More precisely,
\[
B_1 \setminus B =
\{ 2k-1, 2k, \pi(2k-1), \pi(2k) \}.
\] 
Since $ \pi \in \cP_2(2m, 2) $, 
it follows that 
$ \pi(2k-1) = 2t$ and $\pi(2k) = 2s - 1$ 
for some
$ t, s \in [ m ] $.
The components of 
$ \vec{\beta}$
which are not components of 
$\vec{\alpha}$ 
are in this case
$ l_k, j_k, j_{-k}, l_{-k}, l_s, j_{s}, l_{-t}, j_{-t}. $

Since
$ \ou \in \mathcal{A}( \pi, \os ) $, 
we have that
\[
E\big( g_{l_k, j_k } \overline{g_{l_{-t}, j_{-t}}} \big)
\neq 0 
\neq 
E \big( g_{l_s, j_s} \overline{g_{l_{-k}, j_{-k}} }
\big)\]
therefore 
$ l_{-t} = l_k $,
$ l_s = l_{-k} $,
$ j_{-t} = j_k $,
$ j_{s} = j_{-k} $.

On the other hand, 
$ 
(l_k, l_{-k}) = \sigma_k ( i_k, i_{-k}) $
and
\[
i_{-k} = i_{k+1} = \pi_1 \circ \sigma_{k+1}^{-1} (l_{k+1}, l_{-(k+1)} ) .\]
Since
$ l_{k+1}$ and $ l_{-(k+1)} $
are components of 
$ \vec{\alpha} $, 
we have that 
$ i_{-k} $
is uniquely determined by
$ \vec{\alpha}$.

It follows that 
$ \vec{\beta} $
is uniquely determined by
$ i_{k}, j_{k} $
and the components of
$ \vec{\alpha} $.
Therefore
\[
\#( \cA_{ \pi, \os }(B_1))
\leq
M \cdot P \cdot \#( \cA_{ \pi, \os }(B) )
\leq 
\#( \cA_{ \pi, \os }(B))
\cdot
(\max\{ M, P\})^{\frac{1}{2} \#\big( B_k \setminus B\big) }.
\]

Finally, suppose that
$ 2k \in B $
and
$ 2k-1 \notin B $.  
Then 
$ B_1 \setminus B = \{ 2k-1, \pi(2k-1)\} $
and the components of
$ \vec{\beta} $
which are not components of
$ \vec{\alpha} $
are 
$l_k, j_k, l_{-s}, j_{-s} $
where 
$ \pi(2k-1) = 2s $.

As above, 
$ \ou \in \mathcal{A}(\pi, \os) $
gives that
$ l_k = l_{-s} $ 
and
$ j_k = j_{-s} $. 
Moreover, the definition of
$\mathcal{I}(m) $
gives that 
$ j_k = j_{-k} $
which is a component of 
$\vec{\alpha} $. 
Therefore
$\vec{\beta}$
is uniquely determined by 
$ l_k $
and 
$ \vec{\alpha}$,
and the conclusion follows.

The case 
$ 2k \notin B $
and
$ 2k-1 \in B $
is similar.
\end{proof}

The results below concern Equations (\ref{eq:wick}) and (\ref{v:card})
when
$ W_N $ is a 
$ M_N \times M_N $  
Wishart random matrix, that is 
$ W_N = G_N \cdot G_N^\ast $
with
$G_N $ 
a $ M_N \times P_N $ 
Ginibre random matrix, moreover we assume that $\lim_{N \rightarrow \infty} P_N/M_N \rightarrow c$. 

An immediate consequence of  Lemma
\ref{lemma:1:1} is the following.
\begin{cor}\label{cor:1:01}
Suppose that there exists some 
$ D \subseteq [ 2 m ] $ 
such that:
\begin{enumerate}
\item[(i)]if $ l \in D $, then $ \pi(l) \in D $;
\item[(ii)] there exists some
$ k  \in [ m ] $ 
such that 
$ \{ 2k -1, 2k \} \subseteq D $;
\item[(iii)] $ \#( \cA_{\pi, \os} (D))
= o \Big(M_N^{ 1 + \frac{1}{2} \#( D) }\Big) $.
\end{enumerate} 
Then 
\[
\lim_{N \rightarrow \infty}
\mathcal{V} ( \pi, \os ) = 0 .
\]
\end{cor}
\begin{proof}
Let $ D_1  = D $ and, inductively define (by convention, $ 2m + q = q $ ):
\[
D_{ p + 1 } = D_p 
\cup
\big\{
2k + 2p - 1, 2k +2p 
\big\}
\cup
\big\{
\pi(2k + 2p - 1),
\pi( 2k +2p )
\big\}.
\] 
Then 
$ D_m = [ 2m ] $ 
and 
$ \#( \cA_{\pi, \os}(D_m))
= \# (\cA( \pi, \os) )$.

According to Lemma \ref{lemma:1:1}, we have that 
\[
\#( \cA_{\pi, \os}( D_{p+1}))
\leq 
\#( \cA_{\pi, \os}( D_p))
\cdot
( \max\{ M_N, P_N\})^{ \frac{1}{2} \# ( D_{p+1} \setminus D_p ) }  
.
\]

Hence
\[
\# \mathcal{A}(\pi, \os) 
= \# A_{\pi, \os} ( D_{2m}) 
\leq 
\# A_{\pi, \os } ( D )
( \max\{ M_N, P_N\})^{ \frac{1}{2} \# ( [2 m ] \setminus D ) } 
= 
o ( M_N^{ 1 + m }),
\]
and the conclusion follows from equation (\ref{v:card}).
\end{proof}
Corollary \ref{cor:1:01} shall be used to prove the next lemma, needed for the main results of the next section.

\begin{lemma}\label{lemma:2c2}
Suppose that  $\pi \in \cP_2(2m, 2)$ and
$ \pi (2k) = 2k + 1 $
and
$ \pi(2k-1) = 2k + 2 $.


\begin{figure}[t]
\begin{center}
\begin{tikzpicture}[anchor=base, baseline]
\node[below] at (-0.8,0) {$\cdots$};
\node[below] at (0,0) {$g_{l_k\,j_k}$};
\node[below] at (2,0) {$\overline{g_{l_{-k}\,j_{-k}}}$};
\node[below] at (4,0) {$g_{l_{k+1}\,j_{k+1}}$};
\node[below] at (6.5,0) {$\overline{g_{l_{-(k+1)}\,j_{-(k+1)}}}$};
\node[below] at (8.2,0) {$\cdots$};
\node[above] at (0,0) {{\tiny $2 k -1$}};
\node[above] at (2,0) {{\tiny $2 k$}};
\node[above] at (4,0) {{\tiny $2 k+1$}};
\node[above] at (6.5,0) {{\tiny $2 k+2$}};
\draw  (0,0.5) .. controls (0.5,1.5) and (6,1.5) .. (6.5,0.5);
\draw  (2,0.5) .. controls (2.3,1) and (3.7,1) .. (4,0.5);
\draw [dotted] (-0.1,-0.55) .. controls (0.6, -2.1) and (5.8, -2.1) .. (6,-0.55);
\draw [dotted] (2.3,-0.55) .. controls (2.7, -1.6) and (3.6, -1.6) .. (4.2,-0.55);
\draw [dashed] (0.2,-0.55) .. controls (0.7, -1.3) and (1.7, -1.3) .. (2.3,-0.55);
\draw [dotted] (1.8,-0.55) .. controls (2.5, -1.5) and (3.0, -1.5) .. (3.7,-0.55);
\draw [dashed] (4.2,-0.55) .. controls (4.9, -1.3) and (5.9, -1.3) .. (6.8,-0.55);
\end{tikzpicture}
\end{center}
\caption{\label{fig:1} The pairs 
$ (2k -1, 2k + 2) $
 and  
 $ (2k, 2k + 1) $
 from $ \pi $
  are representing by solid lines.}
\end{figure}
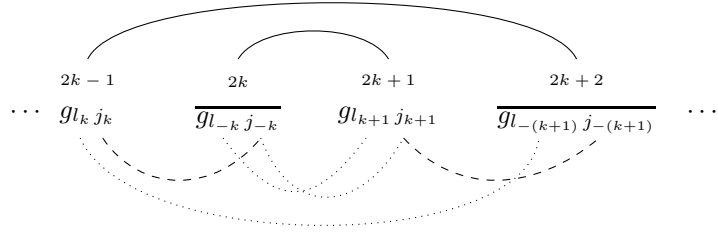

If 
$ \sigma_k $ and $ \sigma_{k+1} $
are such that
\begin{equation}
\lim_{N \rightarrow\infty}
\frac{ 
\#\big\{ (i, j, l)\in [ N ]^3:\ 
\sigma_k (i, j) = \sigma_{k+1}( i, l)
\big\}
}{M_N^2} = 0 
\tag{$\mathfrak{c}.1$}
\end{equation}
then
\[
\lim_{ N \rightarrow \infty }
\mathcal{V}(\pi, \os) = 0.
\]
\end{lemma}
\begin{proof}
Let 
$ D= \{ 2k-1, 2k, 2k + 1, 2k + 2 \} $.
Then, for 
$ \ou \in \mathcal{A}(\pi, \os) $ 
with
\begin{align*}
& \ou = (i_1, j_1, j_{-1}, i_{-1}, \dots, i_m, j_m, j_{-m}, i_{-m})\\
& \os(\ou)  =  (l_1, j_1, j_{-1}, l_{-1}, 
\dots, l_m, j_m, j_{-m}, l_{-m})
\end{align*}
we have that
$ \os ( \ou ) [ D ] 
= ( l_k, j_k, j_{-k }, l_{-k },
l_{k+1}, j_{k + 1},
j_{-(k + 1)}, l_{-(k + 1)}) $.

The condition 
$ \ou \in \mathcal{I}(m) $ 
gives that
$ j_k = j_{-k} $,
$ j_{k+ 1}= j_{-(k+1)} $
(see the dashed lines in Figure 1.)
and that
$ i_{-k} = i_{k + 1} $.
So
$ \os (\ou)[ D ] $ 
is uniquely determined by
$( i_k, j_k,  i_{k +1}, j_{k+1}, i_{-(k+1)}) $.
On the other hand,
the condition
$ \ou \in \mathcal{A}(\pi, \os) $
gives that
\[
E\big( g_{l_k, j_k)} \cdot
\overline{ g_{l_{-(k+1)}, j_{k+1}} }
\big) 
\neq 0 
\neq 
E \big( \overline{g_{ l_{-k}, j_{k} } }
\cdot g_{l_{k+1}, j_{k + 1} }
\big),
\]
that is 
$ j_k = j_{k+1} $
and 
$(l_k, l_{-k}) = ( l_{-(k+1)}, l_{k+1} ) $ 
(see the dotted lines from Figure 1.).
Therefore
\begin{align*}
\#(\cA_{\pi, \os}(D)) \leq  &
\# \big\{ (j_k, i_k, i_{k+1}, i_{-(k+1)} )
:\
(l_k, l_{-k}) = ( l_{-(k+1)}, l_{k+1} )
\big\}\\
& = 
\big\{ (j_k, i_k, i_{k+1}, i_{-(k+1)} )
:\
\sigma_k (i_k, i_{k+1} )
=
\sigma_{k+1} 
( i_{k+1}, i_{-(k+1)} ) 
\big\}\\
& = o( M_N^3) 
= o( M_N^{1 + \frac{1}{2} \#( D )})
\end{align*}
and the conclusion follows from Corollary 
\ref{cor:1:01}.
\end{proof}
\begin{lemma}\label{lemma:1c1}
For 
$ k  \in [ m ] $
and 
$ \pi \in \cP_2(2m, 2) $,
with the convention 
$ 2m + q = q $,
we have that:
\begin{enumerate}
\item[(i)] If 
$ \pi(2k) = 2k + 1 $
but
$ \pi (2k - 1 ) \neq 2k + 2 $
and
\begin{equation}
\lim_{ N \rightarrow \infty}
\frac{\# \big\{
(i, j, l) \in [ N ]^3:\ 
\pi_2\circ \sigma_k (i, j) = 
\pi_1\circ \sigma_{k+1}(j, l) \big\}}
{M_N^3}=0
\tag{$\mathfrak{c}.2$}
\end{equation}
then $ \displaystyle \lim_{N \rightarrow \infty} \mathcal{V}( \pi, \os ) = 0 $.

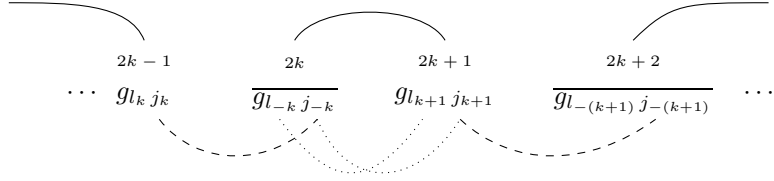
\begin{figure}[h]
	\begin{center}
		\begin{tikzpicture}[anchor=base, baseline]
		\node[below] at (-0.8,0) {$\cdots$};
		\node[below] at (0,0) {$g_{l_k\,j_k}$};
		\node[below] at (2,0) {$\overline{g_{l_{-k}\,j_{-k}}}$};
		\node[below] at (4,0) {$g_{l_{k+1}\,j_{k+1}}$};
		\node[below] at (6.5,0) {$\overline{g_{l_{-(k+1)}\,j_{-(k+1)}}}$};
		\node[below] at (8.2,0) {$\cdots$};
		\node[above] at (0,0) {{\tiny $2 k -1$}};
		\node[above] at (2,0) {{\tiny $2 k$}};
		\node[above] at (4,0) {{\tiny $2 k+1$}};
		\node[above] at (6.5,0) {{\tiny $2 k+2$}};
	     \draw  (0,0.5) .. controls (- 0.2, 1) and (-1, 1) .. (-1.8,1);
	     \draw (6.5, 0.5) .. controls ( 7, 1) and (7.2, 1) ..
	      (8.3, 1);
		\draw  (2,0.5) .. controls (2.3,1) and (3.7,1) .. (4,0.5);
		\draw [dotted] (2.3,-0.55) .. controls (2.7, -1.5) and (3.6, -1.5) .. (4.2,-0.55);
		\draw [dashed] (0.2,-0.55) .. controls (0.7, -1.2) and (1.7, -1.2) .. (2.3,-0.55);
		\draw [dotted] (1.8,-0.55) .. controls (2.5, -1.5) and (3.0, -1.5) .. (3.7,-0.55);
		\draw [dashed] (4.2,-0.55) .. controls (4.9, -1.2) and (5.9, -1.2) .. (6.8,-0.55);
		\end{tikzpicture}
	\end{center}
	\caption{\label{fig:2} The pairing $\pi$ is represented by solid lines.}
\end{figure}
\item[(ii)]  If 
$ \pi(2k-1) = 2k+2 $
but 
$ \pi(2k ) \neq 2k + 1 $
and 
\begin{equation}
\lim_{N \rightarrow \infty}
\frac{\# \big\{
(i, j, l) \in [ N ]^3:\ 
\pi_1\circ \sigma_k (i, j) = 
\pi_2\circ \sigma_{k+1}(j, l) \big\}}
{M_N^3}=0
\tag{$\mathfrak{c}.3$}
\end{equation}
then 
$ \displaystyle \lim_{N \rightarrow \infty} \mathcal{V}( \pi, \os ) = 0 $.

\begin{figure}[h]
	\begin{center}
		\begin{tikzpicture}[anchor=base, baseline]
		\node[below] at (-0.8,0) {$\cdots$};
		\node[below] at (0,0) {$g_{l_k\,j_k}$};
		\node[below] at (2,0) {$\overline{g_{l_{-k}\,j_{-k}}}$};
		\node[below] at (4,0) {$g_{l_{k+1}\,j_{k+1}}$};
		\node[below] at (6.5,0) {$\overline{g_{l_{-(k+1)}\,j_{-(k+1)}}}$};
		\node[below] at (8.2,0) {$\cdots$};
		\node[above] at (0,0) {{\tiny $2 k -1$}};
		\node[above] at (2,0) {{\tiny $2 k$}};
		\node[above] at (4,0) {{\tiny $2 k+1$}};
		\node[above] at (6.5,0) {{\tiny $2 k+2$}};
    \draw  (2,0.5) .. controls (1.5, 1.3) and (-1.5, 1.4) .. (-1.5,1.4);
	     \draw (4, 0.5) .. controls ( 4.5, 1.3) and (7.5, 1.4) ..
	      (8.3, 1.4);
		\draw  (0,0.5) .. controls (0.8,1.2) and (5.2,1.2) .. (6.5,0.5);
		\draw [dashed] (0.2,-0.55) .. controls (0.7, -1.2) and (1.7, -1.2) .. (2.3,-0.55);
		\draw [dashed] (4.2,-0.55) .. controls (4.9, -1.2) and (5.9, -1.2) .. (6.8,-0.55);
		\draw [dotted] (-0.1,-0.55) .. controls (0.6, -2.2) and (5.8, -2.2) .. (6,-0.55);
		\draw [dotted] (0.2,-0.55) .. controls (0.6, -1.8) and (5.8, -1.8) .. (6.8,-0.55);
		\end{tikzpicture}
	\end{center}
	\caption{\label{fig:3}}
\end{figure}
\end{enumerate}
\end{lemma}
\begin{proof}
For part (i), define
\[  D = \{  2k -1, 2k, 2k + 1, 2k + 2 \} \cup \{ \pi( 2k -1), \pi(2k +2 ) \}.
\]
Then 
$ \# D = 6 $ 
and $ \pi(l) \in D $ 
whenever 
$ l \in D $. 
Therefore, according to Corollary \ref{cor:1:01}, it suffices to show that
$ \# A_{\pi, \os} (D) = o(M_N^4) $.

But, for 
$ \ou \in \mathcal{A}(\pi, \os) $,
with the notations from above,
$ \os(\ou) [ D ] $ 
is uniquely determined by
$ (l_k, j_{k}, j_{-k}, l_{-k},
l_{k+1}, j_{k+1}, j_{-(k+1)}, l_{-(k+1)} ) $.
On the other hand, since 
$ \ou \in \mathcal{I}(m) $, 
we have 
$ j_p = j_{-p} $
for each
$ p \in [ m ] $ 
(see the dashed lines in Figure 2.)
and 
$ (l_k, l_{-k}, l_{k+1}, l_{-(k+1)}) $
is uniquely determined by
$(i_k, i_{k+1}, i_{k+2}) $.
So  
$ \ou \in \mathcal{A}(\pi, \os) $
gives that (see the dotted lines in Figure 2.):
\begin{align*}
\# A_{\pi, \os} (D) 
& \leq 
\# \big\{
( j_k, j_{k+1}, i_k,  i_{k+1}, i_{k+2} ) :\ 
E \big(  \overline{g_{l_{-k}, j_k } } \cdot g_{l_{k+1}, j_{k+1}} \big) \neq 0  \big\}\\
& = 
\big\{
( j_k, j_{k+1}, i_k,  i_{k+1},i_{k+2})   
:\ 
( l_{-k}, j_k  )
= (l_{k+1}, j_{k+1})  \big\} \\
& = 
\big\{
( j_k, i_k,  i_{k+1}, i_{k+2} ) :\ 
\pi_2 \circ \sigma_k ( i_k, i_{k+1}) 
= \pi_1 \circ \sigma_{k+1}
(i_{k+1}, i_{k+2})  \big\}\\
& = P_N \cdot o(M_N^3) = 
o ( M_N^4).
\end{align*}
The proof for part (ii) is similar (see Figure 3, with the same conventions as above).
\end{proof}

\begin{lemma}\label{lemma:5}
Suppose that $\pi \in \cP_2(2m, 2)$, 
$ k< p  < t $, 
and that 
$ a, b, c, d $ 
are elements of
$ [ 2m ] $
such that
$ \pi(a) = c $,  
$ \pi(b) = d $ 
and
$ a \in \{ 2k -1, 2k \} $,
$ b \in \{ 2p -1, 2p \} $,
$ c \in \{ 2t-1, 2t \} $,
and
$ d > 2t $ or $ d < 2 k-1 $.
Then
\[
\lim_{ N \rightarrow \infty}
\mathcal{V}(\pi, \os) = 0. 
\]
\end{lemma}   
\begin{proof}
Let
$
\ou = (i_1, j_1, j_{-1}, i_{-1}, \dots, i_m, j_m, j_{-m}, i_{-m}) 
$ 
be an element  from
$
\mathcal{A}(\pi, \os)  
$ 
and
$ \os (\ou) = ( l_1, j_1, j_{-1}, l_{-1}, 
\dots, j_{-m}, l_{-m}).
$ 

Put 
$ D_1 
= \{ 2k -1, 2k \} \cup \{ \pi(2k-1), \pi(2k)\} $.
Then 
$ \# D_1 = 4 $,
since  
$ a $ 
is an element of 
$ \{ 2k-1, 2k \} $
and
$ \pi(a) = c > 2k $.
Moreover, by construction, 
$ \os(\ou)[ D_1] $
is uniquely determined by 
$ (l_k, j_{k}, j_{-k}, l_{-k})$.
But, since 
$ \ou \in \mathcal{I}(m) $,
we also have that
$ j_k = j_{-k} $.
So
\begin{equation}\label{eq:cor:2:1}
\# A_{\pi, \os}(D_1) 
\leq 
\# \big\{
(l_k, j_k, l_{-k}):\
l_{\pm k } \in [ M_N ], 
j_k \in [ P_N ] \big\} 
= M_N^2 \cdot P_N 
\end{equation}
For
$ 1 \leq s \leq p -k -1 $, 
define
\[
D_{s+1} = D_s \cup 
\{ 2(k+s)-1, 2(k+s) \} 
\cup 
\{ \pi(2(k+s) - 1), \pi( 2(k + s ) ) \}.
\]
Note that
$ \{ 2k -1, 2k, \dots, 2p -3, 2p -2  \} \subseteq D_{p-k} $
but
$\{ 2p -1, 2p \} \nsubseteq D_{p-k}$,
since 
$ b \in \{ 2p -1, 2p \} $ 
and
$ \pi(b) \notin D_{p-k} $.

Applying again Lemma \ref{lemma:1:1}, we obtain
\begin{align}\label{eq:5:1}
\#( \cA_{\pi, \os} ( D_{p-k} ) )
\leq &
\# ( \cA_{\pi, \os} (D_1) )
\cdot ( \max\{ M_N, P_N \})^{\frac{1}{2} \# ( D_{p-k} \setminus D_1)}\\
& =
O(M_N^{1 + \frac{1}{2}\#(D_{p-k})} ).
\nonumber
\end{align}

Next, put
\[
F_1 = D_{p-k} \cup \{2t -1, 2t \} 
\cup \{  \pi ( 2 t - 1), \pi( 2t)\}
\]
and remark that
\begin{equation}
\label{eq:5.2}
\#( \cA_{\pi, \os} (F_1) )
= O\Big(M_N^{ 1 + \frac{1}{2} \# (F_1)} \Big).
\end{equation}
To justify (\ref{eq:5.2}), note that if
$ \{ 2t-1, 2t \} \subset D_{p-k} $, 
then
$ F_1 = D_{p-k} $
and Equation (\ref{eq:5.2}) is just Equation (\ref{eq:5:1}).
Suppose that
$
\{ 2t -1, 2t \} \nsubseteq  D_{p-k}. 
$
Since
$ c \in \{ 2t -1, 2t \} $
and 
$ c = \pi(a) \in D_{p-k} $,
we have that
$ \# F_1 = \# D_{p-k} + 2 $, 
so it suffices to show that
\begin{equation}\label{eq:5.3}
\#( \cA_{\pi, \os}(F_1)) \leq 
M_N
\cdot \#(\cA_{\pi, \os}(D_{p-k})).
\end{equation}
But
$ \os(\ou)[ F_1 ] $ 
is uniquely determined by
$ \ou(\ou)[ D_{p-k}] $
and by 
$(l_t, j_t, j_{-t}, l_{-t}) $. 
Since 
$ \ou \in \mathcal{I}(m) $,
we have that 
$ j_t = j_{-t} $.
If 
$ c = 2t -1 $, 
then 
$ l_t $ and $ j_t $ 
are components of 
$ \os(\ou)[ D_{p-k}] $
so 
$ \os(\ou)[ F_1] $
is uniquely determined by $ l_{-k} $ and
$ \os(\ou)[ D_{p-k}] $.
Similarly, if
$ c = 2t $, 
then 
$ \os(\ou)[ F_1] $
is uniquely determined by $ l_{k} $ and
$ \os(\ou)[ D_{p-k}] $. 
So (\ref{eq:5.3}) is proved.

Finally, for
$ 1 \leq s \leq t-p-1 $,
put
\[
F_s+1= F_s 
\cup \{ 2(t-s) -1, 2(t-s)\}
\cup
\{ \pi( 2(t -s ) -1 ), \pi(2 (t -s ) ) \}.
\]   
Lemma \ref{lemma:1:1} and Equation
(\ref{eq:5.2}) give that
\begin{align}\label{eq:5:3}
\#( \cA_{\pi, \os} ( F_{ t - p }))
\leq &
\#( \cA_{\pi, \os} (D_{p-s}) )
\cdot
(\max\{ M_N, P_N\})^{\frac{1}{2} \# (F_{t-p} \setminus D_{p-k})} \\
& = 
O \Big( M_N ^{1 + \frac{1}{2}\# (F_{t-p} )}\Big).
\nonumber
\end{align}
Let 
$ B = F_{ t - p } \cup \{ 2p -1, 2p \} \cup \{ \pi(2p-1), \pi(2p) \} $.
It suffices to show that
\[
\#( \cA_{\pi, \os}(B) ) =
O(M_N^{\frac{1}{2}\# B} )
\]
and the conclusion follows from Corollary 
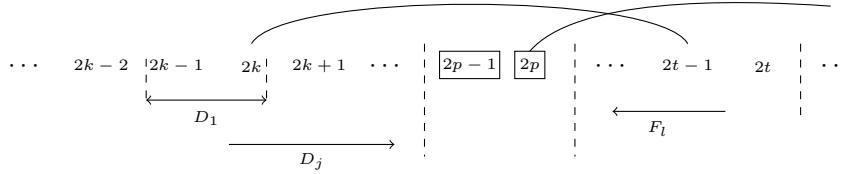
\ref{cor:1:01}.


\begin{figure}[!htb]

\begin{center}
\begin{tikzpicture}[anchor=base, baseline]

\node[above] at (-3.3,0) {$\cdots$};
\node[above] at (-2.3,0) {{\tiny $2 k -2$}};
\node[above] at (-1.3,0) {{\tiny $2 k -1$}};
\node[above] at (-0.3,0) {{\tiny $2 k $}};
\node[above] at (.6,0) {{\tiny $2 k + 1$}};

\node[above] at (1.5,0) {$\cdots$};

\draw[dashed] (2, 0.5) -- (2, - 1);

\draw (2.2, .4) rectangle (3, 0.036);

\node[above] at (2.6, 0) {{\tiny $2p-1$}};

\node[above] at (3.4, 0) {{\tiny $2p$}};

\draw (3.2, .4) rectangle (3.6, 0.036);

\draw[dashed] (4, 0.5) -- (4, - 1);

\node[above] at (4.5,0) {$\cdots$};

\draw[dashed] (7, 0.5) -- (7, - .5);

\node[above] at (5.5,0) {{\tiny $2 t-1$}};
\node[above] at (6.5,0) {{\tiny $2 t$}};

\node[above] at (7.5,0) {$\cdots$};

\draw (-0.3,0.5) .. controls ( 0.1,1.2 ) and (5.1 , 1.2)  .. (5.5, 0.5);

\draw (3.4, 0.4) .. controls ( 4,1.1 )  and (6, 1.1) .. (7.4, 1);

\draw[dashed]  (-0.1, 0.3) -- (-0.1, -0.25);

\draw[dashed]  (-1.7, 0.3) -- (-1.7, -0.25);

\draw[<->]    (-1.7, -0.25) -- (-0.1, -0.25);

\node[above] at (-0.9,-0.7) {{\tiny $D_1$}};

\draw[->]    (-0.6, -0.84) -- (1.6, -0.84);

\node[above] at (0.5,- 1.3) {{\tiny $D_j$}};

\draw[<-]    (4.5, -0.4) -- (6, -0.4);

\node[above] at (5.1,-0.85) {{\tiny $F_l$}};





\end{tikzpicture}
\end{center}
\end{figure}
Since 
$ b \in \{ 2p -1, 2p \} $
and 
$ \pi(b) = d  \notin F_{t-p} $
we have that
$ \{ 2p-1, 2p \} \nsubseteq F_{t-p} $.
Also, note that, by construction, 
$\{ 2(p-1) -1, 2(p-1)\}  $
and
$ \{ 2(p+1)-1, 2(p+1) \} $
are subsets of 
$ F_{t-p} $.
Hence 
$ (l_{p-1}, l_{-(p-1)}) = \sigma_{p-1}(i_{p-1}, i_{-(p-1)})  $
and 
$(l_{p+1}, l_{-(p+1)})= \sigma_{p+1}( i_{p+1}, i_{-(p+1)}) $
are uniquely determined by
$\os(\ou)[ F_{t-p}] $.
Since
$ i_p = i_{-(p-1)} $ 
and
$   i_{-p}= i_{p+1} $
it follows that
$ (l_p, l_{-p})= \sigma_p( i_p, i_{-p}) $
is uniquely determined by
$\os(\ou)[ D_{p-k}] $.
Therefore
$ \os(\ou)[ B ] $ 
is uniquely determined by 
$\os(\ou)[ F_{t-p}] $
and by
$ j_k $, hence
\begin{equation}\label{eq:5:4}
\#( \cA_{\pi, \os} (B)) \leq
P_N \cdot \#( \cA_{\pi, \os} (F_{t-p}))
\end{equation}

Note also that $ \pi(b) = d \notin \{2p-1, 2p\}$, hence
if
$ \{ 2p -1, 2p \} \cap F_{t -p } = \emptyset $,
then
$ \# ( B \setminus F_{t-p}) = 4 $ 
and Equations (\ref{eq:5:3}) and (\ref{eq:5:4}) give that
\[
\#( \cA_{\pi} (B)) = O\big( N ^{2 + \frac{1}{2}\#( F_{t-p})}) 
= O(N^{ \frac{1}{2}\# (B) } \big).
\]
Suppose that 
$ \{ 2p-1, 2p \} \cup F_{t-p}
\neq \emptyset $.
Since
$ b \in \{2p-1, 2p \} $ 
but 
$ b, \pi(b) \notin F_{t-p}$,
it follows that 
$ \#( B) = \# (F_{t-p}) + 2 $.
If 
$ b = 2p-1 $, 
then 
$ 2p \in F_{t-p} $, 
therefore
$ j_{-k}= j_k $ equals a component of 
$ \os(\ou)[F_{t-p}]$. The case $ b = 2p $ is similar. So, in these two cases, we have that
\[ 
\#( \cA_{\pi, \os}(B) )
=  \#( \cA_{\pi, \os}(F_{t-p})) = 
O\big(M_N^{\frac{1}{2}\#(B)}\big),
\]  
and  Corollary 
\ref{cor:1:01} gives the conclusion.
\end{proof}
An immediate consequence of Lemma \ref{lemma:5} above is the following.
\begin{cor}\label{cor:noncrossing}
Let $\pi \in \cP_2(2m,2)$ and let 
$ \delta $ 
be the pair partition  given by 
$ \delta (2k -1) = 2k $. 
If
$ \pi \vee \delta $ 
is crossing, then
\[
\lim_{ N \rightarrow \infty }
\mathcal{V}(\pi, \os) = 0.
\]
\end{cor}
%

Let $ \tau $ be a partition of $ [ n ] $. A block $ B $ of $ \tau $ will be called a \emph{segment} if its elements are consecutive numbers. If $ \tau $ is non-crossing, it is shown in \cite[Remark 9.2 (2)]{nica-speicher} that it contains at least one segment. 
\begin{cor}\label{cor:segment:0}
Suppose that 
$ \pi \in \cP_2(m, 2) $
is such that
$ \displaystyle
\lim_{N \rightarrow \infty}
\mathcal{V}(\pi, \os) \neq 0.
$
Suppose also that
$ S  = ( 2k-1, 2k, \dots, 2p-1, 2p ) $ 
is a segment of 
$ \pi \vee \delta $. 
Then  the restriction of 
$ \pi $ 
to 
$ S $ is one of the following:
\begin{itemize}
\item[(i)] $ \pi( 2k - 1) = 2p $ 
and 
$ \pi(2t)= 2t+1 $ 
for each 
$ t  = k, k+1, \dots, p - 1 $.

\begin{figure}[h!]

\begin{center}
\begin{tikzpicture}[anchor=base, baseline]

\node[above] at (-2,0) {{\tiny $2 k -1$}};
\node[above] at (-1,0) {{\tiny $2 k $}};
\node[above] at (0,0) {{\tiny $2 k + 1$}};
\node[above] at (1,0) {{\tiny $2 k +2$}};
\node[above] at (2,0) {{\tiny $2 k + 3$}};
\node[above] at (3,0) {$ \cdots $ };
\node[above] at (3.5,0) {$ \cdots $ };
\node[above] at (4.5,0) {{\tiny $2 p-2$}};
\node[above] at (5.5,0) {{\tiny $2 p-1$}};
\node[above] at (6.5,0) {{\tiny $2 p$}};

\draw  (-2,0.5) .. controls (-1.7,1.6) and (6.2,1.6) .. (6.5,0.5);

\draw  (-1,0.5) .. controls (-0.75,1) and (-0.25,1) .. (0,0.5);

\draw  (1,0.5) .. controls (1.25,1) and (1.75,1) .. (2,0.5);

\draw  (4.5,0.5) .. controls (4.75,1) and (5.25,1) .. (5.5,0.5);


\end{tikzpicture}

\end{center}
\end{figure}


\item[(ii)]
$ \pi(2k) = 2 p - 1 $ 
and $ \pi(2t-1) = 2( t + 1 ) $
for each
$ t = k, k + 1 , \dots, p-1 $.

\begin{figure}[h!]
\begin{center}
\begin{tikzpicture}[anchor=base, baseline]

\node[above] at (-2,0) {{\tiny $2 k -1$}};
\node[above] at (-1,0) {{\tiny $2 k $}};
\node[above] at (0,0) {{\tiny $2 k + 1$}};
\node[above] at (1,0) {{\tiny $2 k +2$}};
\node[above] at (2,0) {{\tiny $2 k + 3$}};
\node[above] at (3,0) {{\tiny $2 k + 4$}};

\node[above] at (4,0) {$ \cdots $ };
\node[above] at (4.5,0) {$ \cdots $ };
\node[above] at (5.5,0) {{\tiny $2 p-3$}};
\node[above] at (6.5,0) {{\tiny $2 p-2$}};
\node[above] at (7.5,0) {{\tiny $2 p-1$}};
\node[above] at (8.5,0) {{\tiny $2 p$}};

\draw  (-2,0.5) .. controls (-1.7,1.2) and (0.7,1.2) .. (1,0.5);
\draw  (-1,0.5) .. controls (-0.7,1.7) and (7.2,1.7) .. (7.5,0.5);
\draw  (0,0.5) .. controls (0.3,1.2) and (2.7,1.2) .. (3,0.5);

\draw  (2,0.5) .. controls (2.3,1.2)  and (3, 1.2) .. (3.3,1.1);


\draw  (5.5,0.5) .. controls (5.8,1.2) and (8.2,1.2) .. (8.5,0.5);

\draw  (6.5,0.5) .. controls (6.2,1.1)  and (5.5, 1.2) .. (5,1.1);

\end{tikzpicture}
\end{center}
\caption{\label{fig:6}}
\end{figure}

%
\end{itemize}
\end{cor}
\begin{proof}
The result is trivial for 
$ \# S = 2 $. 
If 
$ \# S > 2 $,
it suffices to show that  each two consecutive blocks of 
$\delta $ 
that are contained in $ S $ have elements connected by 
$ \pi $. 

Suppose 
$ q , r \in [m] $ 
are such that
$ r - q > 1 $
and there exists 
$ a \in \{2 q - 1, 2 q \} $
and
$ b \in \{  2 r -1, 2 r \} $
such that
$ \pi( a ) = b $.
In other words we have two non-consecutive  
blocks of $\delta$ that are connected by $\pi$.

Since 
$ \pi \in \cP_2(m, 2) $,
we have that $ a $ and $ b $ have different parities.
Suppose first that 
$ a = 2 q -1 $ 
and $ b = 2r $.
If the set 
$
S_1 = \{ 2q+1, 2 q + 2, \dots, 2p-3, 2p -2 \} $
is invariant under $\pi $, 
then 
$  S $ 
is not a segment of 
$ \pi \vee \delta $.
Since $ \pi $ connects only numbers of different parity, there are at least 2 elements of $ S_1 $, one even and one odd,
whose image under $\pi $ is outside $ S_1 $.
From Lemma \ref{lemma:5}, they can only be connected to elements of 
$\{ 2q-1, 2q, 2r-1, 2r \} $.
But
$ \pi(2q -1) = 2r $, 
so there exist some
$ 2t-1, 2s \in S_1 $ 
such that
$ \pi(2t -1) = 2q $
and 
$ \pi(2s) = 2r -1 $.
In particular, 
$ S_0 = \{ 2q -1, 2q, 2q+1, \dots, 2r \} $
is a block of 
$ \pi \vee \delta $. 
So $ q = k $ and $ r = p $.

If 
$ s < t $, 
then
$ \pi(2t -1) = 2q < 2s-1 $ 
and 
$ \pi (2s) = 2r -1 > 2t $.
Lemma \ref{lemma:5} gives then that
$\displaystyle \lim_{N \rightarrow \infty} \mathcal{V}(\pi, \os) = 0 $.
Hence $ t \leq s $.

If 
$ t \neq q + 1 $,
then
$ S_2 = \{ 2(q+1) -1, 2(q+1), \dots, 2(t-1)-1, 2(t-1) \} $
is non-void. 
Lemma \ref{lemma:5} gives that
\[
\pi(S_2) \subseteq S_1 \cup \{ 2q -1, 2q, 2t-1, 2t\} \]
But 
$ S_2 \neq S $,  
so 
$ \pi(S_2) \neq S_2$. 
Also, 
$ \pi(\{ 2q -1, 2q \} ) = \{ 2t -1, 2r \} $,
therefore
$ \pi(S_2) = S_2 \setminus \{ \pi( 2t )\} $, 
which contradicts the bijectivity of
$ \pi$. 
Same argument gives also that 
$ s+1 = r $.

The argument for the case $ a = 2 q $ and $ b = 2r -1 $ is identical.
\end{proof}

\section{A result on partial transposes of self-adjoint matrices}

Following \cite{mingo-popa-wishart}, we will define partial transposes as follows.
Suppose that  $X $ is a $ bd\times bd $ matrix with entries in some algebra $ \mathcal{A} $.
 We can see $ X $ as a 
 $ b \times b$
  block-matrix,
$ X = [ X_{i, j}]_{i, j=1}^b $,
 with the entries
  $X_{i, j} $
   being 
   $ d\times d $
    matrices over 
    $ \mathcal{A} $.
We denote by $X^{\Gamma(b,d)} $ the \emph{$(b, d)$-partial transpose} of $ X $, that is the matrix obtained by transposing each block of $ X $, but keeping the positions of the blocks:
\[
X^{\Gamma(b, d)} = \big[ X_{i, j}^t \big]_{i, j=1}^b.
\] 

Equivalently, if 
$ bd = M $,
for each
$ (i, j) \in[ M ]^2$
there exist some unique 
$ \alpha_1, \alpha_2 \in [ b ] $
and
$ \beta_1, \beta_2 \in [ d ]  $
such that
\[
(i, j) = \big( (\alpha_1 -1) d + \beta_1,(\alpha_2 -1) d + \beta_2  \big),
\]
i.e. the
$(i, j) $ entry of $ X $
is the 
$ (\beta_1, \beta_2) $
entry 
of the 
$(\alpha_1, \alpha_2) $
block of size 
$ d \times d $ of $ X $.
Then
\begin{align} 
\Gamma(b, d) (i, j) &=  \Gamma (b, d)\big((\alpha_1 -1) d + \beta_1,
(\alpha_2 -1) d + \beta_2  \big)\label{gamma:1}\\
&= \big(  ( \alpha_1 - 1) d  + \beta_2, ( \alpha_2 - 1) d + \beta_1 \big).\nonumber
\end{align} 

\begin{defn}
For 
$ \sigma, \tau \in \mathcal{S}([ M ]^2 ) $
denote
\begin{align*}
\mathfrak{c}(\sigma, \tau) 
& = 
\# \big\{ (i, j) \in [ M ]^2 
:\ \sigma (i, j) = \tau(i, j)
\big\}\\
\mathfrak{j}(\sigma, \tau )
&  = 
\# \big\{ (i, j, l) \in [ M ]^3  :\
\sigma(i, j) = \tau ( i, l) 
\big\}
\end{align*}
\end{defn}

With the notations above, we have the following.
\begin{thm}\label{thm:gamma}
Let $ M $ be a positive integer and
$ b, d, B, D $
be such that 
\[ bd = BD = M. \]
Suppose that 
$ d \leq D $
and let 
$ Q = dL= Dl $
be the least common multiple of 
$ d $ and $ D $.
Then
\[
\frac{M^2}{L^2} \leq 
\mathfrak{c} \big(
\Gamma (b, d), \Gamma(B, D)\big)
=
\mathfrak{j}\big( \Gamma (b, d), \Gamma(B, D)\big)
\leq \frac{M^2}{L}.
\]
\end{thm}

\begin{proof}
Let $ X $ be a 
$ M \times M $ 
matrix. 
Since  both 
$ d $ and $ D $ divide $ M $,
so does their least common multiple 
$ Q $.
So we can see 
$ X $ 
as a
$ M/ Q \times \ M/Q $ 
block matrix, with each block entry a
$ Q \times Q $
matrix.

Fix
$ m_1, m_2 \in [ M / Q ] $ 
and let
\[ 
K(m_1, m_2) = \big\{
\big( (m_1 - 1) Q + \mu_1, (m_2 -1) Q + \mu_2 \big) : \  \mu_1, \mu_2 \in [ Q ]    
\big\}
\]
(i.e. 
$ K(m_1, m_2) $ 
is he set of indices of entries from the $(m_1, m_2) $
block of
size 
$ Q \times Q $ 
of $ X $), and let
\begin{align*}
K_{\mathfrak{c}} (m_1, m_2) & 
= K (m_1 , m_2) \cap
\big\{
(i, j) \in [ M ]^2:\
\Gamma(b, d) (i, j) = \Gamma(B, D) (i, j)
\big\}\\
K_{\mathfrak{j}} ( m_1, m_2) & = 
K(m_1, m_2) \cap \big\{ 
(i, j) \in [ M ]^2:\ 
\Gamma(b, d) (i, j) = \Gamma(B, D) (k, j) \\
& \hskip1.2in\textrm{for some } k \in [ M ]
\big\}.
\end{align*} 
Since
$ \Gamma (b, d ) \big( K(m_1 , m_2) \big)
= \Gamma(B, D ) \big( K(m_1, m_2) \big)
= K(m_1, m_2) $,
it suffices to show that
\begin{equation}\label{k:ineq}
\frac{Q^2}{L^2} = d^2 \leq
\# K_{\mathfrak{c}}(m_1, m_2)
= \# K_{\mathfrak{j}} (m_1, m_2)
\leq d^2L = \frac{Q^2}{L}.
\end{equation} 
For the first inequality in (\ref{k:ineq}), note that, for 
$ \mu_1, \mu_2 \in [ d] \subseteq [ D ] $, we have that  
\[
\big( (m_1 - 1)Q +\mu_1, (m_1 -1)Q + \mu_2 \big)
= 
\big( (m_1 - 1)L \cdot d +\mu_1, (m_1 -1)L\cdot d + \mu_2 \big)
\]   
and equation (\ref{gamma:1}) gives that
\[
\Gamma(b, d) (i, j) =    
\big( (m_1 - 1)L \cdot d +\mu_2, (m_2 -1)L\cdot d + \mu_1 \big)
= \Gamma(M/Q, Q)(i, j).
\]
Similarly, since 
$ d \leq D $,
we have that
\begin{align*} 
\Gamma(B, D) (i, j)= &
\Gamma(B, D)
\big( ( m_1 - 1) l \cdot D + \mu_1,  (m_1 - 1) l \cdot D + \mu_2 \big)\\
& 
= \big( ( m_1 - 1) l \cdot D + \mu_2,  (m_1 - 1) l \cdot D + \mu_1 \big)  = \Gamma(M/Q, Q)(i, j) 
\end{align*}
therefore
\[
\big\{
\big( (m_1 - 1) Q + \mu_1, (m_2 -1) Q + \mu_2 \big) : \  \mu_1, \mu_2 \in [ d ]    
\big\} \subseteq K_{\mathfrak{c}},
\] 
which gives the first inequality in (\ref{k:ineq}).

We shall prove the equality 
$
\# K_{\mathfrak{c}}(m_1, m_2)
= \# K_{\mathfrak{j}}(m_1, m_2)
$ 
in two steps. 
First, for
$ m_1 = 1 $ and $ m_2 =1 $,
let
\begin{align*}
(i, j) & = 
\big(
(\alpha_1 -1)d + \beta_1, 
(\alpha_2 -1)d + \beta_2
\big)\\
(k, j) &=
\big(
( \gamma_1 -1)D + \delta_1,
( \gamma_2 -1 )D + \delta_2
\big)
\end{align*}
with
$ \alpha_1, \alpha_2\in [ L ] $,
$ \beta_1, \beta_2 \in [ d ] $,
$ \gamma_1, \gamma_2 \in [ l ] $
and 
$ \delta_1, \delta_2 \in [ D ]$.

Then
$ \Gamma( b, d) (i, j) = \big( 
(\alpha_1 -1)d + \beta_2),
(\alpha_2 -1) d + \beta_1) \big) $
which is point on the vertical
line passing through 
$\big( (\alpha_1 -1)d + \beta_2),
(\alpha_2 -1)d + \beta_2 ) \big) $.
Similarly,
$ \Gamma( B, D)(k, j) $
is a point 
on the vertical line passing through
$\big( (\gamma_1 - 1) D + \delta_2,
(\gamma_2 - 1) D + \delta_2
\big).$
But
\[  (\alpha_2 -1) d + \beta_2 =
j
= (\gamma_2 - 1) D + \delta_2,
\]
so the equality
$ \Gamma(b, d) (i, j) = \Gamma (B, D)(k, j) $
gives that
\[
(\alpha_1 -1) d + \beta_2 
= 
( \gamma_1 -1) D + \delta_2.
\]
Hence
$ (\alpha_1 - \alpha_2 ) d = (\gamma_1 - \gamma_2 ) D $,
which, since 
$ \alpha_1, \alpha_2 \in [ L ] $
and
$ \gamma_1, \gamma_2 \in [ l ] $,
gives that
$ \alpha_1 = \alpha_2 $
and 
$ \gamma_1 = \gamma_2 $.
So
$ \Gamma(b, d) (i, j) = (j, i) $
and 
$ \Gamma(B, D) (k, j) = (j, k)$,
which gives 
$ i= k $.

Furthermore, since 
$ \alpha_1 = \alpha_2 $,
note that we have also obtained
\begin{equation}\label{k:c:3}
K_{\mathfrak{c}}(1, 1) \subseteq 
\big\{ 
\big(  ( \alpha -1)d + \beta_1, ( \alpha -1) d + \beta_2 \big)
: \  \alpha \in [ L ] , \beta_1 , \beta_2 \in [ d ] 
\big\}
\end{equation}

Next, for arbitrary 
$ m_1, m_2 \in [ N / Q ] $
and
$(i, j), (k, j) \in K(m_1, m_2) $,
we have that
$  i =  (m_1 -1) Q + i^\prime $,
$ j =   (m_2 -1) Q + j^\prime $
and
$ k =  (m_1 - 1) Q + k^\prime $
with
$ i^\prime, j^\prime, k^\prime \in [ Q ] $, 
that is
$ (i^\prime, j^\prime), (k^\prime, j^\prime) \in K (1, 1) $.
Moreover,
\begin{align*}
& \Gamma(b, d)(i, j)  =
( (m_1 - 1)Q, (m_2-1)Q) 
+ \Gamma(b, d) (i^\prime, j^\prime)\\
& \Gamma (B, D)(k, j) =
( (m_1 - 1) Q,  (m_2-1) Q )
+ \Gamma(B, D ) (k^\prime, j^\prime)
\end{align*}
and the conclusion follows from the case
$ m_1 = m_2 =1 $.

For the third part of (\ref{k:ineq}),  let
$(i , j) \in K_{\mathfrak{c}}(m_1, m_2) $.
The argument above gives that  
\[
(i, j) = \big( (m_1 - 1) Q, ( m_2 -1 ) Q \big)  
+ ( i^\prime, j^\prime) 
\] 
for some 
$ (i^\prime, j^\prime ) \in K_{ \mathfrak{c}}(1, 1) $.
Therefore, according to (\ref{k:c:3}),
\[
\#K_{\mathfrak{c}}(m_1, m_2) \leq 
\# \big\{ ( \alpha, \beta_1, \beta_2 ):\  \alpha\ in [ L ], 
\beta_1, \beta_2 \in [d ] \big\}   = \frac{Q^2}{L}
\]
which gives the conclusion.
\end{proof}
\begin{thm}\label{thm:2:3}
Suppose that
$ \big\{ b_N \big\}_N $,
$ \big\{ d_N \big\}_N $,
$ \big\{ B_N \big\}_N $,
$ \big\{ D_N \big\}_N $
and
$ \big\{ M_N \big\}_N $
are sequences of positive integers such that
$ \big\{ M_N \big\}_N $
is strictly increasing 
and
\[
M_N = b_N \cdot d_N = B_N \cdot D_N .\]
Denote 
$ \sigma_N = \Gamma ( b_N, d_N ) $
and
$ \tau_N = \Gamma ( B_N, D_N ) $.
\begin{enumerate}
\item[(i)] If 
$ \displaystyle 
\lim_{N \rightarrow \infty}
D_N = \infty $, 
then
\[
\# \big\{ (i, j, k) \in [M_N]^3:\
\pi_2 \circ \sigma (i, j) = 
\pi_2 \circ \tau (k,j) 
\big\} = o(M_N^3).
\]
\item[(ii)] If
$ \displaystyle
\lim_{N \rightarrow \infty}
B_N = \infty $,
then
\[
\# \big\{ (i, j, k) \in [M_N]^3:\
\pi_1 \circ \sigma (i, j) = 
\pi_1 \circ \tau (k, j ) 
\big\} = o(M_N^3).
\]
\end{enumerate}
\end{thm}
\begin{proof}
Let
$ \alpha_1, \alpha_2 \in [ b_N ] $,
$ \beta_1, \beta_2 \in [ d_N ] $,
$ \gamma_1, \gamma_2 \in [ B_N ] $
and
$ \delta_1, \delta_2 \in [ D_N ] $ 
be such that
\begin{align*}
i & = 
(\alpha_1 - 1) d_N + \beta _1 \\
j & =
(\alpha_2 -1) d_N + \beta_2 = (\gamma_2 -1) D_N + \delta_2\\
k & = (\gamma_1 - 1) D_N + \delta_1.
\end{align*}
In particular, the couple 
$ (\gamma_2, \delta_2) $
is uniquely determined 
by
$ (\alpha_1, \alpha_2, \beta_1, \beta_2 ) $.

First, suppose that
$ \pi_2 \circ \sigma (i, j) = 
\pi_2 \circ \tau (k,j)
$. 
The condition  is equivalent to
\[(\alpha_2 - 1)d_N + \beta_1 = 
(\gamma_2 -1) D_N + \delta_1,
\]
hence 
$ \delta_1 $
is  uniquely determined by
$(\alpha_1, \alpha_2, \beta_1, \beta_2) $. 
So
\begin{align*}
\# \big\{ (i, j, k) \in  [M_N]^3 & :\
\pi_2 \circ \sigma (i, j) = 
\pi_2 \circ \tau (k,j) 
\big\}\\
\leq  &
\big\{
(\alpha_1, \alpha_2, \beta_1, \beta_2, \gamma_1):\
\alpha_1, \alpha_2 \in [ b_N ], \beta_1, \beta_2 \in [ d_N ], \gamma_1 \in [ B_N ] 
\big\}\\
= & b_N^2 d_N^2 B_N = \frac{M_N^3}{D_N}. 
\end{align*}
But
$ \displaystyle  \frac{M_N^3}{D_N}
= o(M_N^3) $
if
$ \displaystyle 
\lim_{N \rightarrow \infty} D_N = \infty $
and part (i) follows.   

Next, suppose that
$ \pi_1 \circ \sigma (i, j)
= \pi_1 \circ \tau(k, j) $.
The condition is equivalent to
\[
(\alpha_1 -1 ) d_N + \beta_1 
= (\gamma_1 - 1 ) D_N + \delta_2, 
\]
hence
$ \gamma_1 $ 
is now uniquely determined by
$(\alpha_1, \alpha_2, \beta_1, \beta_2). $
Therefore
\begin{align*}
\# \big\{ (i, j, k) \in [M_N ]^3 & :\ 
\pi_1 \circ \sigma (i, j)
= \pi_1 \circ \tau(k, j)
\big\}\\
\leq &
\# 
\big\{
(\alpha_1, \alpha_2, \beta_1, \beta_2, \delta_1) :\
\alpha_1, \alpha_2 \in [ b_N ], 
\beta_1, \beta_2 \in [ d_N], 
\delta_1 \in [ D_N ] 
\big\}\\
= & b_N^2 d_N^2 D_N = \frac{M_N^3}{B_N},
\end{align*} 
and
$ \displaystyle 
\frac{M_N^3}{B_N} = o(M_N^3) $
if
$ \displaystyle 
\lim_{N \rightarrow \infty}  B_N = \infty $
hence part (ii) follows. 
\end{proof}

\section{Main results on asymptotic freeness}\label{main:section}

\subsection{Partial transposes}
\begin{lemma}\label{lemma:nf:1}
Suppose that 
$ \big\{ M_N \big\}_N $
is a strictly increasing sequence of positive integers and that
$ W_N $
is a $ M_N \times M_N $
Wishart random matrix 
for each $ N $.

If 
$ \sigma_N $ and $ \tau_N $
are symmetric permutations from
$ \mathcal{S}([ M_N ]^2) $
such that
\[
\liminf_{N \rightarrow \infty}
\frac{\# \big\{ 
(i, j) \in [ M_N ]^2:\ \sigma_N(i, j) = \tau_N (i, j)
\big\}}{N^2} > 0
\]
then 
$ W_N^{\sigma_N } $ 
and 
$ W_N^{\tau_N } $
are \emph{not} asymptotically (as 
$ N \rightarrow \infty $) 
free. 
\end{lemma}
\begin{proof}
We have that
\begin{align*}
\kappa_2 (
W_N^{\sigma_N}, W_N^{\tau_N } )
& = 
E \circ \tr \big(
W_N^{\sigma_N} \cdot W_N^{\tau_N } \big) 
- E \circ \tr \big(
W_N^{\sigma_N} \big) \cdot 
E \circ \tr \big(W_N^{\tau_N }\big)\\
= & \sum_{
\substack{i_1, i_2 \in [ M_N],\\
j_2, j_2 \in [ P_N ]}}
\frac{1}{N} 
E \big( g_{ l_1, j_1}
\overline{ g_{l_{-2}, j_2} }
\big) 
\cdot
E\big( \overline{g_{l_{-1}, j_1}}
g_{l_2, j_2}
\big)\\
= &
\sum_{
\substack{i_1, i_2 \in [ M_N],\\
j_2, j_2 \in [ P_N ]}}
\frac{1}{N^3} \delta_{(l_1, l_{-1})}^{(l_{-2}, l_2)} \delta_{j_1}^{j_2}
\end{align*}
where
$(l_1, l_{-1}) = \sigma_N(i_1, i_2) $
and 
$(l_2, l_{-2}) = \tau_N (i_2, i_1) $.
Using that 
$ \sigma_N $ and $\tau_N $ 
are symmetric, we obtain
\[ 
\kappa_2 
\big(W_N^{\sigma_N}
\cdot W_N^{\tau_N} \big)
= \frac{P_N}{M_N} 
\cdot
\frac{ \# \{ (i_1, i_2) :\
\sigma_N(i_1, i_2) = \tau_N(i_1, i_2) \}}{M_N^2}
\]
So 
$ \displaystyle 
\lim_{N \rightarrow \infty} 
\kappa_2 (W^\sigma \cdot W^\tau) \neq 0 .
$
\end{proof}

An immediate consequence of Lemma \ref{lemma:nf:1} and Theorem \ref{thm:gamma} is the following.

\begin{cor}\label{cor:nf:1}
Suppose that 
$ \big\{ b_N \big\}_N $,
$ \big\{ d_N \big\}_N $,
$ \big\{ B_N \big\}_N $,
and
$ \big\{ D_N \big\}_N $,
are sequences of positive integers 
such that
$ d_N \leq D_N $, 
that
$ b_N \cdot d_N = B_N \cdot D_N = M_N$
and the sequence
$ \big\{ M_N \big\}_N $
is strictly increasing.

Define 
$ Q_N = d_N \cdot L_N $ 
to be the least common multiple of 
$ d_N $ and $ D_N $.

Suppose also that
$ W_N $
is a 
$ M_N \times M_N $
Wishart matrix for each $ N $.
If the sequence
$ \big\{ L_N \big\}_N $
is bounded, then
$  W_N^{\Gamma(b_N, d_N)} $
and
$ W_N^{\Gamma(B_N, D_N)} $
are \emph{not} asymptotically free.
\end{cor}
\begin{proof}
Theorem \ref{thm:gamma} gives that
\[
\liminf_{N \rightarrow \infty}
\frac{\# \big\{ 
(i, j) :\
\Gamma(b_N, d_N)(i, j)
= \Gamma(B_N, D_N)(j, i)
\big\}}{M_N^2}
\geq
\liminf_{N \rightarrow \infty}
\frac{1}{L_N^2} > 0,
\]
and Lemma \ref{lemma:nf:1} implies the conclusion.
\end{proof}
A particular case of Corollary \ref{cor:nf:1}
is given below.
\begin{cor}\label{cor:nf:2}
With the notations from Corollary \ref{cor:nf:1}, if either both sequences 
$ \big\{ b_N \big\}_N $ and 
$ \big\{ B_N \big\}_N $
or both sequences
$ \big\{ d_N \big\}_N $
and 
$ \big\{ D_N \big\}_N $
are bounded, then
$  W_N^{\Gamma(b_N, d_N)} $
and
$ W_N^{\Gamma(B_N, D_N)} $
are \emph{not} asymptotically free.
\end{cor} 
\begin{proof}
If condition (i) holds true, then 
$ \big\{ L_N \big\}_N $ 
is bounded above by 
$ \displaystyle
\lim_{N \rightarrow \infty} D_N $, 
which is finite.
If condition (ii) holds true, then
$ \big\{ L_N \big\}_N $ 
is bounded above by 
$ \displaystyle
\lim_{N \rightarrow \infty} B_N
$
which is finite.
\end{proof}

To prove the main result of this Section, Theorem \ref{thm:main}, we will utilize Lemmata
\ref{segment:1} and \ref{segment:2} below. To simply the notations in their statements and proofs, we will introduce some new notations.

First, identifying $ 2m + q $ to $ q $,
let 
$ \nu_1 \in \cP_2 ( 2m, 2)  $ 
be given by
$ \nu_1 (2 k ) = 2 k + 1 $
for $ k \in [ m ] $,
i.e.
\[ \nu_1 = \big( (1, 2m), (2, 3), (4, 5), \dots, 
(2m-2, 2m -1)
\big). \]
For 
$ m \geq 3 $,
let
$ \nu_2 \in \cP_2(2m, 2) $
be given by
$ \nu_2 ( 2k -1 ) = 2 k  + 2 $
for 
$ k \in [ m  ] $, 
i.e.
\[
\nu_2 = \big( (1, 4), (3, 6), 
\dots, (2m-3, 2m), (2m-1, 2)
\big).
\]

Next, denote by 
$ J ( m) $ the set 
\begin{align*}
J(m) = \{ 
( j_1, j_2, i_2, i_3, j_3, j_4, i_4, i_5, \dots,  i_{2m-1},  j_{2m-1},  & j_{2m})
:\
j_l \in [ P_N ],  i_k \in [ M _N ]\\
&  \textrm{ and } j_{2s -1} = j_{2s}, 
i_{2s} = i_{2s + 1}\}
\end{align*}
(i.e. the elements of 
$ \cJ $ 
are elements of 
$ \mathcal{I}(m) $
without the first and last component). 

Finally, if 
$ \ou = ( j_1, j_{-1}, i_{-1}, i_2, j_2,   j_{-2}, i_{-2},  \dots, i_{2m-1}, j_{2m-1}, j_{2m} )$  is an element from
$ J ( m ) $ 
and
$ a, b  \in [ N ] $,
denote 
\[
(a, \ou, b )  = ( a, j_1, j_{2}, i_{2}, i_3, j_3,    \dots, i_{2m-1}, j_{2m-1}, j_{2m}, b ) 
\in \mathcal{I}(m).
\]
With the notations from Section
\ref{section:vV} and from above  we have then the following results.
\begin{lemma}\label{segment:1}
Suppose that
$ \sigma = \Gamma(b, d) $ 
and
$ \os = ( \sigma, \sigma, \dots, \sigma)
\in 
\mathcal{S}([M]^2)^m $,
with
$ M = bd $.
For 
$ a, b \in [ M ] $
we have that
\begin{itemize}
\item[(i)] If
$ a \neq b $ 
then,   for every 
$ \ou \in \cJ $,
\[ v( \nu_1, \os, (a, \ou, b)) = 
v( \nu_2, \os, (a, \ou, b)) = 0 
\]
\item[(ii)] If
$ m $  is odd, then
\begin{align*}
\sum_{ \ou \in \cJ } v(\nu_1, \os, (a, \ou, a))=
\frac{P}{M} \cdot d^{1-m}\\
\sum_{ \ou \in \cJ } v(\nu_2, \os, (a, \ou, a))=
\frac{P}{M} \cdot b^{1-m}
\end{align*}
\item[(iii)] If $ m $ is even, then
\begin{align*}
\sum_{ \ou \in \cJ } v(\nu_1, \os, (a, \ou, a))=
\frac{P}{M} \cdot d^{2-m}\\
\sum_{ \ou \in \cJ } v(\nu_2, \os, (a, \ou, a))=\frac{P}{M} \cdot b^{2-m}
\end{align*}
\item[(iv)]
$\displaystyle
\lim_{N \rightarrow \infty}
\sum_{ \pi \in\{ \nu_1, \nu_2 \} 
}
\sum_{ \ou \in \cJ }
v(\pi, \os, (a, \ou, a)) = 
\lim_{N \rightarrow \infty}
\kappa_m(W^\sigma, W^\sigma, \dots, W^\sigma)
$
\end{itemize}
\end{lemma}
\begin{proof}
Let 
$ \ou \in \cJ $, given by
$ \ou = ( j_1, j_2, i_2, i_3, j_3, j_4, i_4, i_5, \dots, i_{2m-1}, j_{2m-1},\ab j_{2m}) .
$
As before, we can identify each 
$ i_k $
to a pair
$(\alpha_k, \beta_k) \in [b] \times [d] $
via
\[
i_k = (\alpha_k -1) d + \beta_k .
\]
Also, put
\begin{align*}
a &= (\alpha_1 -1) d + \beta_1\\
b & = (\alpha_{2m} - 1) d + \beta_{2m}.
\end{align*}

Denote
\begin{align*}
\oj  & = (j_1, j_2, \dots, j_{2m})\\
\vec{\alpha} &
= ( \alpha_1, \alpha_2, \dots, \alpha_{2m-1}, \alpha_{2m})\\
\vec{\beta} &
= ( \beta_1, \beta_2, \dots, \beta_{2m-1}, \beta_{2m}).
\end{align*}
The condition
$ \ou \in \cJ $
is equivalent to
\begin{equation}
\label{eq:segment:1}
\left\{
\begin{array}{l }
\textrm{
$ \oj $ 
is constant on the cycles of  
}\\
\hspace{3cm}
\theta = \big( (1, 2), (3, 4),
\dots, (2m-1, 2m) \big) 
\\
\textrm{
$ \vec{\alpha} $
and
$ \vec{\beta} $
are constant on the cycles of
}\\
\hspace{3cm}
\nu_0 = \big(  (1), (2, 3), (4, 5),
\dots, (2m-2, 2m-1),  (2m) \big). 
\end{array}
\right.
\end{equation}
The condition
$ v( \nu_1, \ou, (a, \ou, b) ) \neq 0  $
is equivalent to
\begin{equation}
\label{eq:segment:2}
\left\{
\begin{array}{l }
\textrm{
$ \oj $ 
and 
$ \vec{\alpha} $
are constant on the cycles of 
$ \nu_1 $
}\\
\textrm{
$ \vec{\beta} $
is constant on the cycles of
$ \nu_2 $.
}\\
\end{array}
\right.
\end{equation}

Hence, a set of necessary and sufficient condition for both
$ \ou \in \cJ $
and
$ v( \nu_1, \ou, (a, \ou, b) ) \neq 0  $
is
\begin{equation}
\label{eq:segment:3}
\left\{
\begin{array}{l }
\textrm{
$ \oj $ 
is constant on the cycles of
$ \nu_1 \vee \theta $}\\
\textrm{
$ \vec{\alpha} $
is constant on the cycles of 
$ \nu_1 \vee \nu_0 $
}\\
\textrm{
$ \vec{\beta} $
is constant on the cycles of
$ \nu_2 \vee \nu_0 $.
}\\ 
\end{array}
\right.
\end{equation}
But
$ \nu_1 \vee \theta  = \mathds{1}_{2m} $
and
$ \nu_1 \vee \nu_0  = \nu_1 $.
Moreover, if
$ m $ is odd, 
then 
$ \nu_2 \vee \nu_0 = \mathds{1}_{2m} $;
if
$ m $ is even, 
then
$ \nu_2 \vee \nu_0 $ 
has 2 cycles, 
one containing 
$ \{ 
4k, 4k + 1:\ k \in [ \frac{m}{2}]
\} $
and one containing
$ \{
4k + 2, 4k + 3:\ k \in [ \frac{m}{2}]
\} $. 

If 
$ m $ is even,
(\ref{eq:segment:3})
is therefore equivalent to
\begin{equation}
\label{eq:segment:4:even}
\left\{
\begin{array}{l }
\textrm{
$ \oj $ 
is constant, i.e. 
$ j_1 = j_2 = \dots = j_{2m} $ } \\
\textrm{
$ \alpha_1 = \alpha_{2m} $ 
and 
$ \alpha_{2p } = \alpha_{2p + 1} $
for 
$ p \in [ m-1] $
}\\
\textrm{
$ \beta_1 =  \beta_{4k+ 1} = \beta_{4k} $
and
$ \beta_2 = \beta_{4k + 2} = \beta_{4k+3}$
for
$ k \in [ \frac{m}{2} ] $.}
\end{array}
\right.
\end{equation}
In particular, since 
$ 2m = 4 \cdot \frac{m}{2} $,
the relations from 
(\ref{eq:segment:4:even})
give that
$ \alpha_1 = \alpha_{2m} $ 
and
$ \beta_1 = \beta_{2m} $, 
i. e. 
$ a = b $.
Moreover, 
\begin{align*}
\sum_{\ou \in \cJ} &
v(\nu_1, \os,  (a, \ou, a))
= 
\frac{1}{M^{m}}
\# \big\{ \ou :\ 
\ou \textrm{satisfies (\ref{eq:segment:4:even})}
\big\}\\
= &
\frac{1}{M^{m}}
\# \big\{ (j,  \beta_2, \alpha_2, \alpha_4, \dots, \alpha_{2m-2})
:\
j \in [ P ],  \beta_2 \in [ d ] , 
\alpha_{2t} \in [ b ] , t \in [ m-1]
\big\}\\
=  & 
\frac{1}{M^{m}} \cdot
P b^{m-1} d
=\frac{P}{M}d^{2-m}
\end{align*}

If 
$ m $ is odd,
(\ref{eq:segment:3})
is therefore equivalent to
\begin{equation}
\label{eq:segment:4:odd}
\left\{
\begin{array}{l }
\textrm{
$ \oj $ 
is constant} \\
\textrm{
$ \alpha_1 = \alpha_{2m} $ 
and 
$ \alpha_{2p } = \alpha_{2p + 1} $
for 
$ p \in [ m-1] $
}\\
\textrm{
$ \vec{\beta} $ is constant, i.e.
$ \beta_1 =  \beta_{2} =  \cdots =\beta_{2m} $.
}
\end{array}
\right.
\end{equation}
As above, 
(\ref{eq:segment:4:odd}) 
implies that
$ \alpha_1 = \alpha_{2m} $
and 
$ \beta_1 = \beta_{2m} $, 
i.e.
$ a = b $. 
Also,   
\begin{align*}
\sum_{\ou \in \cJ}
v( & \nu_1, \sigma,  (a, \ou, a))
= 
\frac{1}{M^{m}}
\# \big\{ \ou :\ 
\ou \textrm{\ satisfies (\ref{eq:segment:4:odd})}
\big\}\\
= &
\frac{1}{M^{m}}
\# \big\{ (j,  \alpha_2, \alpha_4, \dots, \alpha_{2m-2})
:\
j \in [ P ], 
\alpha_{2t} \in [ b ] , t \in [ m-1]
\big\}\\
=  & 
\frac{1}{M^{m}} \cdot P
b^{m-1} 
=
\frac{P}{M} \cdot d^{1-m}.
\end{align*}

On the other hand, the condition
$ v( \nu_2, \ou, (a, \ou, b) ) \neq 0  $
is equivalent to 
\begin{equation}
\label{eq:segment:5}
\left\{
\begin{array}{l }
\textrm{
$ \oj $ 
and 
$ \vec{\beta} $
are constant on the cycles of 
$ \nu_1 $
}\\
\textrm{
$ \vec{\alpha} $        
is constant on the cycles of
$ \nu_2 $.
}\\ 
\end{array}
\right.
\end{equation}
Hence the argument from above gives that
$ v( \nu_1, \ou, (a, \ou, b) ) = 0  $          
unless
$ a = b $
and  
\[ 
\sum_{\ou \in J(N, m)}
v(  \nu_2, \sigma,  (a, \ou, a))
= 
\left\{
\begin{array}{l l}
\displaystyle  \frac{P}{M} \cdot b^{2-m} &
\textrm{ if 
$ m $ is even }\\
& \\
\displaystyle
\frac{P}{M} \cdot b^{1-m} &
\textrm{if $ m $ is odd }.
\end{array}
\right.
\] 
\end{proof}

\begin{lemma}
\label{segment:2}
Suppose that
$ \big\{ b_N \big\}_N $,
$ \big\{ d_N \big\}_N $
and
$ \big\{ M_N \big\}_N $
are sequences of positive integers 
such that
$ M_N = b_N \cdot D_N $
and 
$ \big\{ M_N \big\}_N $
is an increasing sequence of positive integers. 
Let 
$ W_N $ 
be a 
$ M_N \times M_N $ 
Wishart random matrix 
and
$ \sigma_N = \Gamma( b_N, d_N ) $.
Then, for any 
$ a \in [ M_N ] $, we have that
\begin{equation}\tag{$\mathfrak{c}.4$}
\lim_{N \rightarrow \infty}
\kappa_m(W_N^{\sigma_N},
W_N^{\sigma_N}, \dots, W_N^{\sigma_N})
=
\lim_{N \rightarrow \infty}
\sum_{ \pi \in\{ \nu_1, \nu_2 \} 
}
\sum_{ \ou \in \cJ }
v(\pi, \os{_N},
(a, \ou, a)) 
\end{equation}
where
$ \vec{\sigma}_N  = ( \sigma_N, \sigma_N, \dots, \sigma_N )$.

In particular, 
$ W_N^{\Gamma(b_N, d_N)} $
has limit distribution if and only if
both
$ \displaystyle
\lim_{N \rightarrow \infty} b_N $
and
$ \displaystyle
\lim_{N \rightarrow \infty} d_N $
exist.
\end{lemma}
\begin{proof}
For the first part, note  first that the definition of free cumulants give that

\[
\lim_{N \rightarrow \infty}
\kappa_m ( W_N^{\sigma_N}, \dots, W_N^{\sigma_N})
= 
\lim_{N \rightarrow \infty}
\sum_{\substack{ \pi \in \cP_2 ( m, 2) \\
\pi\vee \delta = \mathds{1}_{2m}}}
\sum_{\ou \in \mathcal{I}(m)}
\frac{1}{M_N}
v( \pi, \os, \ou) .
\]  
From Corollary \ref{cor:segment:0} , we have that
\[
\big\{ \pi \in \cP_2(m, 2):\ \pi \vee \delta = \mathds{1}_{2m} 
\textrm{ and } 
\lim_{N \rightarrow \infty}
\mathcal{V}(\pi, \os) \neq 0
\big\}
= \big\{ \nu_1, \nu_2 \big\}.
\]
So the definition of
$ \cJ $ 
gives that
\[
\lim_{N \rightarrow \infty}
\kappa_m ( W_N^{\sigma_N}, \dots, W_N^{\sigma_N})
= 
\lim_{N \rightarrow \infty}
\sum_{\pi \in \{ \nu_1, \nu_2 \}}
\frac{1}{M_N}
\sum_{a \in [ N ]} 
\sum_{\ou \in \cJ} 
v( \pi, \os, ( a, \ou, a) )
\] 
and the conclusion follows since Lemma
\ref{segment:1} 
gives that, for
$ \pi \in \big\{ \nu_1, \nu_2 \big\} $,
the value of
$ \displaystyle
\sum_{\ou \in J(N, m)} 
v( \pi, \os, ( a, \ou, a) )
$
does not depend on $ a $.

The second part is an immediate consequence of the first part and of the results (ii) and (iii) in  Lemma \ref{segment:1}.
\end{proof} 
\begin{remark}\label{rem:3:5}
For every 
$ a \in [ N ]$,
with the notations from Lemmata \ref{segment:1} and \ref{segment:2},
if 
$ \pi \in \cP_2(2m, 2) $ 
and
$ \pi \vee \delta = \mathds{1}_{2m} $, 
then
\[
\sum_{\ou \in \cJ } 
v( \pi, \os, ( a, \ou, a) ) =
\mathcal{V}(\pi, \os).
\]
\end{remark}
\begin{proof}
Since
$
\mathcal{I}(m) = \{ (a, \ou, a): \  
a \in [ M_N ], \ou \in \cJ \}
$,
we have that
\[
\mathcal{V}(\pi, \os) =
\frac{1}{M_N}
\sum_{a\in[ M_N ]}
\sum_{ \ou \in \cJ }
v( \pi, \os, (a, \ou, a)).
\]
But
$ \displaystyle 
\sum_{\ou \in \cJ } 
v( \pi, \os, ( a, \ou, a) )
$
does not depend on $ a $,
so the conclusion follows.
\end{proof}

The main result of this Section is Theorem \ref{thm:main}. For convenience, its proof is split into two parts.
The first part, Lemma \ref{lemma:split} below, will be also used in the next Section.  

\begin{lemma} \label{lemma:split}
Suppose that
$ \big\{ M_N \big\} $ is a strictly increasing sequence of positive integers and that for each 
$ N $, 
$ W_N $ 
is a
$ M_N \times M_N $
Wishart random matrix 
and
$ \big\{
\tau_{k, N}:\  k =1,2, \dots, n
\big\} $
is a family of symmetric permutations from 
$ \mathcal{S}([ M_N ]^2) $ such that
\begin{itemize}
\item[(i)] the limit distribution of 
$ W_N^{ \tau_{k, N}} $ exists for every
$ k \in [ n ] $.
\item[(ii)] the families 
$\big\{ \tau_{k, N} \}_N $ 
satisfy the conditions
$ ( \mathfrak{c}.1 ) $, 
$ ( \mathfrak{c}.2 ) $,
$ ( \mathfrak{c}.3 ) $,
$ ( \mathfrak{c}.4 ) $
and the property from Lemma \ref{segment:1}$ ($i$)$.
\end{itemize}
Then the family 
$ \big\{ W_N^{ \tau_{k, N}}:\ 
k =1, 2, \dots, n
\big\} $
is asymptotically free.
\end{lemma}
\begin{proof}
We shall show  that all mixed free cumulants in
$ \big\{
W_N^{ \tau_{k, N}}:\
k = 1, 2, \dots, n
\big\}$ 
do vanish asymptotically. 
As before, to simplify the notations, without danger of confusion, we shall omit the index $ N $.    

Let
$ m $
be a positive integer. 
For 
$ \gamma \in NC(m ) $,
define
$ \widehat{\gamma} \in NC(2m) $ 
as follows. If
$ ( t_1, t_2, \dots, t_q ) $ 
is a block of 
$ \gamma $, 
then
$(2t_1 -1, 2t_1, 2t_2-1, 2 t_2, \dots, 
2 t_q -1, 2 t_q  ) $
is a block of 
$ \widehat{\gamma} $.
Note that, by construction,  
$ \gamma \mapsto \widehat{\gamma} $
is a  bijection from
$ NC(m) $
to 
$ \{ \rho \in NC(2m):\  \rho \vee \delta = \rho \} $.

Let 
$ \os = (\sigma_1, \sigma_2, \dots, \sigma_m) $
with 
$ \sigma_1, \dots, \sigma_m \in 
\big\{
\tau_k  :\  k =1, 2, \dots, n 
\big\} $. 
As seen in Section \ref{section:vV}, 
\[
E \circ \tr 
\big(
W^{\sigma_1} \cdot W^{\sigma_2} \cdots
W^{\sigma_m}
\big) = 
\sum_{\pi \in \cP_2(2m, 2)} 
\mathcal{V}(\pi, \os).
\]
Since each
$ \Gamma ( b_i, d_i) $
is symmetric, Corollary \ref{cor:noncrossing} gives that
$
\displaystyle
\lim_{N \rightarrow \infty}
\mathcal{V}( \pi, \os) = 0 $
unless
$ \pi \vee \delta $ 
is non-crossing. Therefore
\begin{equation}\label{eq:thm:main:4}
\lim_{N \rightarrow \infty}
E \circ \tr 
\big(
W^{\sigma_1} \cdot W^{\sigma_2} \cdots
W^{\sigma_m}
\big) = 
\sum_{\gamma \in NC(m)}
\lim_{N \rightarrow\infty}
\sum_{
\substack{\pi \in \cP_2(2m, 2)\\
\pi \vee \delta = \widehat{\gamma}}
}
\mathcal{V}(\pi, \os).
\end{equation}

It suffices though to prove the following results:
\begin{itemize}
\item[(a.1)] If 
$ \pi \in \cP_2( 2m, 2) $ 
is such that
$ \displaystyle
\lim_{N \rightarrow \infty}
\mathcal{V}( \pi, \os) \neq 0 $
and
$ B = (l_1, l_2, \dots, l_q) $
is a block of 
$ \gamma $,
where
$ \widehat{\gamma} = \pi \vee \delta $,
then
$ \sigma_{l_1}= \sigma_{l_2}
= 
\dots= \sigma_{l_q}$.
\item[(a.2)] For every 
$ \gamma \in NC(m) $, 
we have that
\[
\lim_{N \rightarrow \infty}
\sum_{
\substack{ \pi \in \cP_2(2m, 2)\\
\pi \vee \delta = \widehat{\gamma}}
} 
\mathcal{V}( \pi, \os)
= \prod_{\substack{B = \textrm{block of } \gamma\\
B =(l_1, l_2 \dots, l_q )}}
\lim_{N \rightarrow \infty}
\kappa_{q} 
\big(
W^{\sigma_{l_1}}, W^{\sigma_{l_1}}, \dots,
W^{\sigma_{l_1}}
\big).
\]
\end{itemize}

If 
$ \gamma $ 
has a single block, then
$ \widehat{\gamma} = \mathds{1}_{2m} $.
If
$ m =1 $,
then (a.1) is trivial.
If 
$ m =2 $,
then
$ \pi = \{ (1, 4), (2, 3) \} $.
If
$ \sigma_1 \neq \sigma_2 $, then 
Lemma \ref{lemma:2c2}
imply that
$ \displaystyle
\lim_{N \rightarrow \infty}
\mathcal{V}(\pi, \os) = 0.
$          

If
$ m \geq 3 $,
then Corollary \ref{cor:segment:0}, gives that
$ \displaystyle
\lim_{N \rightarrow \infty}
\mathcal{V}(\pi, \os) = 0.
$
unless 
$ \pi \in \{ \nu_1, \nu_2 \} $.
Suppose first that
$ \pi = \nu_1 $.
In particular, for any two consecutive blocks
$(2t-1, 2t )$ 
and 
$ (2t +1, 2t+2) $
of
$\delta$,
we have that
$ \pi(2t) = 2 t + 1 $
and
$ \pi(2t -1) \neq 2 t + 2 $.
If 
$ \sigma_t \neq \sigma_{t+1} $,
Lemma \ref{lemma:1c1}(i) give that
$ \displaystyle
\lim_{N \rightarrow \infty}
\mathcal{V}(\pi, \os) = 0 $.
Similarly, if
$ \pi = \nu_2$,
then
$ \pi(2t -1) = 2t + 2 $
and
$ \pi(2t) \neq 2t+1 $.
In this case, 
Lemma \ref{lemma:1c1}(ii) give that
$ \displaystyle
\lim_{N \rightarrow \infty}
\mathcal{V}(\pi, \os) = 0 $.

So if 
$ \gamma = \mathds{1}_{2m} $.,
(a.1) is proved.
To show (a.2), note that, from  Corollary \ref{cor:segment:0},
\begin{align*}
\lim_{N \rightarrow \infty}
\sum_{
\substack{ \pi \in \cP_2(2m, 2)\\
\pi \vee \delta = \mathds{1}_{2m}}
}  &
\mathcal{V}( \pi, \os)
= 
\sum_{\pi \in \{ \nu_1, \nu_2\}}
\lim_{N \rightarrow \infty}
\mathcal{V}( \pi, \os)\\
= & 
\lim_{N \rightarrow \infty}
\sum_{\pi \in \{ \nu_1, \nu_2\}}
\frac{1}{ M_N }
\sum_{ \ou \in \cI }
v(\pi, \os, \ou)\\
= & 
\sum_{\pi \in \{ \nu_1, \nu_2\}} 
\lim_{N \rightarrow \infty}
\frac{1}{M_N}
\sum_{a \in [ M_N ] }
\sum_{ \ou \in \cJ }
v(\pi, \os, (a,\ou, a))    
\end{align*}       
and the result follows since the family
$\{ \tau_k:\  k = 1, 2, \dots, n \} $ 
satisfies property
$ ( \mathfrak{c}.4 ) $.    

For the induction step, fix
$ \gamma \in NC(m) $. 
As shown in \cite{nica-speicher},
$\gamma$ 
has at least one block which is a segment.  Without restricting the generality (via a circular permutation) we can suppose that
$B = (1, 2, \dots, q) $
is a block in 
$\gamma $, 
that is
$ \gamma = \mathds{1}_{q} \oplus \omega $.
Then
$ \widehat{\gamma} = \mathds{1}_{2q} \oplus
\widehat{\omega} $
and if
$ \pi \in \cP_2(2m, 2) $
is such that
\[
\pi \vee \delta = \widehat{\gamma} =
\mathds{1}_{2q} \oplus \widehat{\omega},
\]
then
$ \pi = \nu \oplus \rho $
for some
$ \nu \in \cP_2(2q, 2) $
and
$ \rho \in \cP_2(2(m-q), 2) $
such that
$ \nu \vee \delta = \mathds{1}_{2q} $
and
$ \rho \vee \delta = \widehat{\omega} $. 

As shown above, 
$ \nu \vee \delta = \mathds{1}_{2q} $
implies that
$ \sigma_t = \sigma_{t+1} $ 
for
$ t = 1, 2, \dots, q-1$,
and property (a.1) follows from the induction hypothesis.    

To show (a.2), note that, by construction, for each
$ \ou \in \mathcal{I}(m) $ 
there are some unique
$ a, b \in [ N ]$,
$ \vec{v} \in J( q) $
and
$ \vec{w} \in J( m-q) $
such that
$ \ou =
( a, \vec{v}, b, b, \vec{w}, a ) $.
Moreover, denoting
$ \vec{\sigma}^\prime
= ( \sigma_1, \sigma_2, \dots, \sigma_q)$
and
$ \vec{\sigma}^{\prime\prime}
= ( \sigma_{q+1}, \sigma_{q+2}, \dots, \sigma_m)$
we have that
\begin{align*}
v( \pi, \os, \ou) = &
v( \nu \oplus \rho,
\vec{\sigma}^\prime
\oplus
\vec{\sigma}^{\prime\prime}, 
(a, \vec{v}, b, b, \vec{w}, a )\\
= &
v(\nu,  \vec{\sigma}^\prime, (a, \vec{v}, b) )
\cdot
v( \rho,
\vec{\sigma}^{\prime\prime},
( b, \vec{w}, a )  
).
\end{align*}    
Hence, if
$ \ou \in \cI $
is such that
$ v( \pi, \os, \ou) \neq 0 $, 
then Lemma \ref{segment:1}(i) gives that
$ a = b $.
Denoting 
$ m -q = r $, 
we have then:   
\begin{align*}
\mathcal{V}(\pi, \os)
& =
\mathcal{V}(\nu \oplus \rho,
\vec{\sigma}^\prime \oplus \vec{\sigma}^{\prime\prime})\\
= &
\frac{1}{M_N}
\sum_{a\in[ M_N ]}
\sum_{\vec{v} \in 
J( q)}
\sum_{\vec{w} \in J( r)}  
v(\nu,  \vec{\sigma}^\prime, 
(a, \vec{v}, a) )
\cdot
v( \rho,
\vec{\sigma}^{\prime\prime},
( a, \vec{w}, a )  
)\\
=&
\frac{1}{M_N} 
\sum_{( a, \vec{w}, a ) 
\in \mathcal{I}( r ) }
v( \rho,
\vec{\sigma}^{\prime\prime},
( a, \vec{w}, a )) 
\cdot
\sum_{\vec{v} \in J( q )}
v(\nu,  \vec{\sigma}^\prime, 
(a, \vec{v}, a) ).
\end{align*} 
Therefore
\begin{align*}
\sum_{\substack{\pi \in \cP_2(2m, 2)\\
\pi \vee \delta = \widehat{\gamma}}}
\mathcal{V}(\pi, \os)
& =
\sum_{\substack{\rho \in \cP_2(2r, 2)\\
\rho \vee \delta = \widehat{\omega}}}
\sum_{\substack{\nu \in \cP_2(2q, 2)\\
\pi \vee \delta = \mathds{1}_{2q}}} 
\mathcal{V}(\nu \oplus \rho,
\vec{\sigma}^\prime \oplus \vec{\sigma}^{\prime\prime}) \\
=&
\sum_{\substack{\rho \in \cP_2(2r, 2)\\
\rho \vee \delta = \widehat{\omega}}}
\frac{1}{M_N} 
\sum_{( a, \vec{w}, a ) 
\in \mathcal{I}( r) }
v( \rho,
\vec{\sigma}^{\prime\prime},
( a, \vec{w}, a ))\\
& \hspace{3.6cm}
\cdot
\big[
\sum_{\substack{\nu \in \cP_2(2q, 2)\\
\pi \vee \delta = \mathds{1}_{2q}}}
\sum_{\vec{v} \in J( q)}
v(\nu,  \vec{\sigma}^\prime, 
(a, \vec{v}, a) )             
\big] .                      
\end{align*}     
But
\[
\frac{1}{M_N}  
\sum_{( a, \vec{w}, a ) 
\in \mathcal{I}( r) }
v( \rho,
\vec{\sigma}^{\prime\prime},
( a, \vec{w}, a ))=
\mathcal{V}( \rho, \vec{\sigma}^{\prime\prime} )
\]
hence, from Remark \ref{rem:3:5},
\[
\sum_{\substack{\pi \in \cP_2(2m, 2)\\
\pi \vee \delta = \widehat{\gamma}}}
\mathcal{V}(\pi, \os)
=
\big[ 
\sum_{\substack{\nu \in \cP_2(2q, 2)\\
\pi \vee \delta = \mathds{1}_{2q}}}
\mathcal{V}
( \nu, \vec{\sigma}^{\prime})
\big]
\cdot
\big[
\sum_{\substack{\rho \in \cP_2(2r, 2)\\
\rho \vee \delta = \widehat{\omega}}}
\mathcal{V}( \rho, \vec{\sigma}^{\prime\prime})                    \big]
\]
and property (a.2) follows by induction.
\end{proof}
\begin{thm}\label{thm:main}
Suppose that 
$ \{ M_N \}_N $
is a strictly increasing sequence of positive integers and that for each
$ i =1, 2, \dots, n $
there exist two sequences
$ \{ b_{i, N}\}_N $ 
and 
$ \{ d_{i, N}\}_N $
such that:
\begin{itemize}
	\item[(i)] $ b_{1, N}  \leq b_{2, N}  \leq \dots \leq b_{n, N} $
\item[(ii)] for each $ i =1, 2, \dots, n $, we have that $ b_{i, N} \cdot d_{i, N}= M_N $
\item[(iii)] for each $ i =1, 2, \dots, n $, the limits 
$ \displaystyle
\lim_{N \rightarrow \infty} 
b_{i, N} $,
$ \displaystyle
\lim_{N \rightarrow \infty} 
d_{i, N} $ exist. 
\end{itemize}

Then the family
$ \big\{ W_{N}^{\Gamma(b_{i, N}, d_{i, N})}:\
i = 1, 2, \dots, n
\big\}$
is almost surely asymptotically free
if and only if the following condition is satisfied:
\begin{itemize}
\item[($\mathfrak{w}.1$)] If
$ k < l $
and
$ Q_{k, l, N}= d_{k, N} \cdot L_{k, l, N} $
is
the least common multiple of
$ d_{k, N} $ and $ d_{l, N} $,
then
$ \displaystyle
\lim_{N \rightarrow \infty}
L_{k, l, N} = \infty.
$
\end{itemize}
\end{thm}
\begin{proof}
The proof of the almost sure part of this Theorem depends on Theorem \ref{thm:52}. For the reader's convenience we defer the proof of Theorem \ref{thm:52} to Section \ref{sec:fluctuations}.

Suppose the family
$ \big\{ W_{N}^{\Gamma(b_{i, N}, d_{i, N})}:\
i = 1, 2, \dots, n
\big\}$
is almost surely asymptotically free then the family
$ \big\{ W_{N}^{\Gamma(b_{i, N}, d_{i, N})}:\
i = 1, 2, \dots, n
\big\}$
is asymptotically free and Corollary \ref{cor:nf:1} shows that
($\mathfrak{w}.1$) holds.

Suppose now that condition ($\mathfrak{w}.1$)
holds true.
Since all partial transposes are symmetric and since, according to 
Lemmata \ref{segment:1} and \ref{segment:2}, conditions ($\mathfrak{c}.3$) and
($\mathfrak{c}.4$)
are satisfied by partial transposes, we only need to show that 
($\mathfrak{w}.1$) implies that the family
$ \big\{ W_N^{\Gamma(b_{i, N}, d_{i, N})}:\
i = 1, 2, \dots, n
\big\}$
satisfies the conditions 
($\mathfrak{c}.1$), ($\mathfrak{c}.2$) 
and ($\mathfrak{c}.3$).

Let 
$ k <l $.
%
Theorem \ref{thm:gamma} and
($\mathfrak{w}.1$) give that
\begin{equation*}
\mathfrak{j}\big( \Gamma(b_{k, N}, d_{k, N}), 
\Gamma(b_{l, N}, d_{l, N}) \big)
\leq \frac{ M_N^2 }{L_{k, l, N }}
= o( M_N^2 ),
\end{equation*}
i.e. the condition 
( $\mathfrak{c}.1$)
is satisfied.

On the other hand, note that ($\mathfrak{w}.1$)
implies that
$ \displaystyle
\lim_{N \rightarrow \infty}
d_{l, N} = \infty 
$.
Indeed, if
\[
\lim_{N \rightarrow \infty} d_{k, N}
\leq 
\lim_{N \rightarrow \infty} d_{l, N}
< \infty
\] 
then 
$ \displaystyle 
\lim_{N \rightarrow \infty}
Q_{l, k, N} $
is the least common multiple of
$ \displaystyle 
\lim_{N \rightarrow \infty}
d_{k, N} $
and
$ \displaystyle 
\lim_{N \rightarrow \infty}
d_{l, N} $,
which is finite, hence so is
$ \displaystyle 
\lim_{N \rightarrow \infty}
L_{k, l, N} $. 
Condition
($ \mathfrak{w}.1$)
also implies that
$ \displaystyle
\lim_{N \rightarrow \infty }  
b_{k, N}   = \infty $, 
since
\[  b_{k, N} \cdot L_{k, l, N}
\leq M_N = b_{k, N} \cdot d_{k, N} \] 
so
$ \displaystyle
\lim_{N \rightarrow \infty} L_{k, l, N} 
\leq 
\displaystyle
\lim_{N \rightarrow \infty }  
b_{k, N}$.
Theorem \ref{thm:2:3} gives then
\begin{align*}
\# \big\{ (a, b, c):\
\pi_2 \circ \Gamma(b_{k, N}, d_{k, N}) 
(a, b) = 
\pi_2 \circ \Gamma(b_{l, N}, d_{l, N})
(c,b) 
\big\} = o(M_N^3)
\\
\# \big\{ (a, b, c) :\
\pi_1 \circ \Gamma(b_{k, N}, d_{k, N})
(a, b) = 
\pi_1 \circ \Gamma(b_{l, N}, d_{l, N})
(c, b ) 
\big\} = o( M_N^3 ),
\end{align*}           
i.e.  the family 
$ \big\{ \Gamma(b_{i, N}, d_{i, N}):\
i = 1, 2, \dots, n
\big\}$
also satisfy conditions 
($\mathfrak{c}.2$)
and
($\mathfrak{c}.3 $).  So by Lemma \ref{lemma:split} 
the family 
$ \big\{ W_N^{\Gamma(b_{i, N}, d_{i, N})}:\
i = 1, 2, \dots, n
\big\}$ is asymptotically free. In Theorem \ref{thm:52}
we will show that for all $i_1, \dots , i_m$ the mixed moments 
\[
\tr \big( 
   W_N(\omega)^{ \Gamma(b_{i_1, N}, d_{i_1, N})} 
    W_N(\omega)^{ \Gamma(b_{i_2, N}, d_{i_2, N})} 
      \cdots
       W_N(\omega)^{ \Gamma(b_{i_m, N}, d_{i_m, N}) } 
 \big)
\]
converge almost surely as $N \rightarrow \infty$. By what we have shown
above the limit to which they converge is the mixed moment of some free random variables.  Thus the family
$ \big\{ W_{N}^{\Gamma(b_{i, N}, d_{i, N})}:\
i = 1, 2, \dots, n
\big\}$
is almost surely asymptotically free.    
\end{proof}
Utilizing Theorem \ref{thm:gamma}, we can reformulate Theorem \ref{thm:main} in a more intuitive fashion, as shown below.
\begin{cor}\label{cor:49}
Suppose that 
$ \{ M_N \}_N $
is a increasing sequence of positive integers and that for each
$ i =1, 2, \dots, n $
there exist two sequences
$ \{ b_{i, N}\}_N $ 
and 
$ \{ d_{i, N}\}_N $
such that:
\begin{itemize}
\item[(i)] $ b_{i, N} \cdot d_{i, N} = M_N $
\item[(ii)] the limits 
$ \displaystyle
\lim_{N \rightarrow \infty} 
b_{i, N} $,
and 
$ \displaystyle
\lim_{N \rightarrow \infty} 
d_{i, N} $ 
exist.
\end{itemize}

Then the family
$ \big\{ W_N^{\Gamma(b_{i, N}, d_{i, N})}:\
i = 1, 2, \dots, n
\big\}$
is almost surely asymptotically free
if and only if whenever 
$ k \neq l $, we have that
\[
\lim_{N \rightarrow \infty}
\frac{1}{M_N^2}
\#\big\{ (i, j):\ 
\Gamma(b_{k, N}, d_{k, N})(i, j) = 
\Gamma( b_{l, N}, d_{l, N}) (i, j)
\big\} = 0. 
\]
\end{cor}

We conclude this Section with the statements of two particular cases of Theorem \ref{thm:main}.
\begin{cor}
Suppose that the sequences
$ \{ b_{i, N}\}_N $ 
and 
$ \{ d_{i, N}\}_N $
satisfy the following conditions:
\begin{itemize}
\item[(i)]
$ b_{i, N} \cdot d_{i, N} = M_N $
for every
$ i =1, 2, \dots, n $
and the sequence
$ \{ M_N \}_N $ 
is strictly increasing.
\item[(ii)]
$ \displaystyle 
\lim_{N \rightarrow \infty}
b_{1, N} $ 
does exist and
$ \displaystyle
\lim_{N \rightarrow \infty }
\frac{b_{i, N}}{b_{i+1, N}} = \infty
$
for every 
$ i =1, 2, \dots, n-1 $.    
\end{itemize}
Then the family
$ \big\{ W_N^{\Gamma(b_{i, N}, d_{i, N})}:\
i = 1, 2, \dots, n
\big\}$
is almost surely asymptotically free.    
\end{cor}

\begin{cor}
Suppose that, for 
$ i =1, 2, 3 $,
the sequences
$ \{ b_{i, N}\}_N $ 
and 
$ \{ d_{i, N}\}_N $
satisfy the following conditions:
\begin{itemize}
\item[(i)]
$ b_{i, N} \cdot d_{i, N} = M_N $
for every
$ i =1, 2, 3 $ and the sequence 
$ \{ M_N \}_N $ 
is strictly increasing.
\item[(ii)] the limits 
$ \displaystyle
\lim_{N \rightarrow \infty}
b_{1, N} $
and
$ \displaystyle
\lim_{N \rightarrow \infty}
d_{3, N} $ 
exist and are finite,
while  \\         
$ \displaystyle
\lim_{N \rightarrow \infty} b_{2, N}
=   
\lim_{N \rightarrow \infty} d_{2, N} = \infty $.                 
\end{itemize}
Then the family
$ \big\{
W_N^{\Gamma(b_{1, N}, d_{1, N})},
W_N^{\Gamma(b_{2, N}, d_{2, N})},
W_N^{\Gamma(b_{3, N}, d_{3, N})}
\big\}$
is almost surely asymptotically free. 
\end{cor}

An example of a family a permutations non-related to partial transposes that is satisfying Lemma \ref{lemma:split} is given in Remark \ref{remark:sn} below. In the next Section, we shall discuss the framework of
``left-partial transposes", and improve the results from \cite{mingo-popa-wishart}.

\begin{remark}\label{remark:sn}
Suppose that for each positive integer 
$ N $, $ \theta_N $ 
is a permutation from 
$ \mathcal{S}(N ) $ 
such that
\begin{equation}
\label{theta}
\lim_{N \rightarrow \infty}
\frac{ \# \big\{ j \in [ N ]:\  
\theta_N (j) = j 
\big\}}{N} = 0.
\end{equation}
Define 
$ \sigma_N \in \mathcal{S}([ N ]^2) $
via
$ \sigma_N (i, j) =
\big( \theta_N(i), \theta_N(j) \big) $.

If $ W_N $ 
is a 
$ N \times N $
Wishart matrix, then
$ W_N $ 
and
$W_N^{ \sigma_N } $
are almost surely asymptotically free. 
\end{remark}
\begin{proof}
Note first that
$ \big\{ W_N^{ \sigma_N } \big\}_N  $
is an Wishart ensemble,  since, with the notations from Section \ref{section:vV}, 
\[ 
W_N^{ \sigma_N} = G_N^{ \sigma_N } \cdot ( G_N^{ \sigma_N})^\ast.
\]
So condition ($\mathfrak{c}.4$) and the property from Lemma \ref{segment:1}(i) are satisfied, hence it suffices to show that the family
$ \big\{ W_N, W_N^{ \sigma_N} \big\} $
satisfies the conditions
($\mathfrak{c}.1$), 
($\mathfrak{c}.2$)
and
($\mathfrak{c}.3$).

First, note that 
$ \sigma_N (i, j) = (i, j) $ 
is equivalent to 
$ \theta_N(i) = i $ 
and 
$ \theta_N(j) = j $.
So condition (\ref{theta}) gives
\[
\#\big\{ (i, j) \in [ N ]^2:\  \sigma_N (i, j) = (i, j) \big\} 
= \# \big\{ i \in [ N ]:\  \theta_N(i) = i \big\}^2 = o(N^2). 
\]
Similarly,
\begin{multline}
\#\big\{(i, j, l) \in [ N ]^3:\ 
\pi_2 \circ \sigma_N (i, j) = \pi_1 (j, l)
\big\} \\
=
\#\big\{(i, j, l) \in [ N ]^3:\ 
\theta_N (j ) = j 
\big\}
= o(N^3).
\end{multline}
Finally,
\begin{align*}
\# \big\{ 
(i, j, l) \in [ N ]^3 :\
\pi_1 \circ \sigma_N (i, j) 
= \pi_2 ( j, l)
\big\} 
= 
\big\{(i, j, l) \in [ N ]^3 :\ 
\theta_N ( i) = l 
\big\}
= N^2
\end{align*}
and the conclusion follows from Lemma \ref{lemma:split} and Theorem \ref{thm:52}. 
\end{proof}

\subsection{Relation
to left partial transposes }  
For 
$ b \cdot d = M $, 
we define the 
\emph{left partial transpose}
$\Ltr(b, d)(X) $
of a 
$ M \times M $
matrix $ X $ 
as the transpose of 
$ \Gamma(b,d)(X) $.
I.e. we see 
$ X $ 
as a  
$ b \times b $
block matrix 
$ X = \big[  X_{i, j} \big]_{i, j =1}^b $
with entries
$ X_{i, j} $
being
$ d \times d $
matrices;
the matrix 
$ \Ltr(b, d)(X) $ 
is obtained then by changing 
each 
$ X_{i, j} $ 
with
$ X_{j, i}$
without transposing the entries inside the blocks.   

\begin{thm}\label{thm:ltr:main}
Suppose that  
the sequences 
$\{ M_N \}_N $
$ \{ b_{N }\}_N $, 
$ \{ d_N\}_N $,
$ \{ B_{N} \}_N $
and
$ \{ D_N \}_N $
satisfy the following conditions:
\begin{itemize}
	\item[(i)]
	 $ \{ M_N \}_N $ 
	is strictly increasing;
	\item[(ii)] $ b_{ N} \cdot d_{ N}
	= B_N \cdot D_N = M_N $ 
	\item[(ii)] the limits 
	$ \displaystyle
	\lim_{N \rightarrow \infty} 
	b_{ N} $,
	$ \displaystyle
	\lim_{N \rightarrow \infty} 
	d_{ N} $ ,
	respectively
	$ \displaystyle
	\lim_{N \rightarrow \infty} 
	B_{ N} $, 
	and
	$ \displaystyle
	\lim_{N \rightarrow \infty} 
	D_{ N} $          
	do exist.
\end{itemize} 	
Suppose also that
$ W_N $ 
is a
$ M_N \times M_N $
Wishart matrix.
Then the following statements are equivalent:
\begin{itemize}
	\item[(1)] 
	 $ \displaystyle
	\lim_{N \rightarrow \infty}
	d_N \cdot B_N 
	= 
	\lim_{N \rightarrow \infty}
	b_N \cdot D_N 
	= \infty $
	\item[(2)] 
	$ W_N^{ \Gamma(b_{ N}, d_{ N})} $
	and
	$ W_N^{ \ltr(B_{ N}, D_{ N})} $
	are almost surely asymptotically free
	\item[(3)]
	$ \displaystyle
	\lim_{N \rightarrow \infty}
	\frac{ \# \big\{ (i, j) \in [ M_N ]^2 :\ 
		\Gamma( b_N, d_N ) (i, j) = \Ltr ( B_N, D_N ) (i, j) 
		\big\}}{M_N^2} = 0.
	$
\end{itemize}
	\end{thm}
\begin{proof}
By Theorem \ref{thm:52},  $ W_N^{ \Gamma(b_{ N}, d_{ N})} $
	and
	$ W_N^{ \ltr(B_{ N}, D_{ N})} $
	are almost surely asymptotically free if and only if they are
	are asymptotically free. So to prove the theorem we shall just
	work with asymptotic freeness.

	Suppose that 
	$ \displaystyle
	\lim_{N \rightarrow \infty}
	d_N \cdot B_N 
	= 
	\lim_{N \rightarrow \infty}
	b_N \cdot D_N 
	= \infty
	$.   
	Since 
	$W_N^{ \Ltr( B_N, D_N )}$
	is the transpose of
	$ W_N^{ \Gamma( b_N, d_N )} $, 
	it does satisfy the condition
	 ($\mathfrak{c}.4$) as well as the property from Lemma \ref{segment:1}(i).
	So, to show that property (2) holds true, it suffices then to show that
	$ \Gamma(b_N, d_N) $
	and
	$ \Ltr( B_N, D_N ) $
	satisfy the conditions 
	($\mathfrak{c}.1$),
	($\mathfrak{c}.2$),
	($\mathfrak{c}.3$).

	As seen before, for each 
	$ (i, j, l) \in [ M_N ]^3$
	there some are unique 
	$ \alpha_1, \alpha_2 \in [ b_N ]$,
	$ \beta_1, \beta_2 \in [ d_N ] $,
	$ \gamma_1, \gamma_2 \in [ B_N ]$
	and
	$ \delta_1, \delta_2 \in [ D_N ]$
	such that
	\begin{align*}
	i & = ( \alpha_1 - 1 ) d_N + \beta_1\\
	j & = ( \alpha_2 -1) d_N + \beta_2 
	= ( \gamma_2 -1) D_N + \delta_2 \\
	l & = ( \gamma_1 -1) D_N + \delta_1.
	\end{align*}
	The condition
	$  \Gamma( b_N, d_N ) (i, j) 
	= 
	\Ltr( B_N, D_N) (l, j) $
	gives then
	\begin{equation*}
	\left\{
	\begin{array}{l}
	( \alpha_1 - 1 ) d_N + \beta_2 
	=
	( \gamma_2 -1) D_N + \delta_1\\
	( \alpha_2 -1) d_N + \beta_1 
	= 
	( \gamma_1 -1) D_N + \delta_2\\
	( \alpha_2 -1) d_N + \beta_2 
	= ( \gamma_2 -1) D_N + \delta_2
	\end{array}
	\right.
	\end{equation*}
	which is equivalent to
	\begin{equation}\label{eq:ltr:3.1}
	\left\{
	\begin{array}{l}
	( \alpha_1 - \alpha_2 ) d_N
	=
	\delta_1 - \delta_2 
	\\
	\beta_1  - \beta_2
	= 
	( \gamma_1 - \gamma_2) D_N 
	\\
	( \alpha_2 -1) d_N + \beta_2 
	= ( \gamma_2 -1) D_N + \delta_2.
	\end{array}
	\right.
	\end{equation}

	If
	$ \displaystyle 
	\lim_{N \rightarrow \infty} d_N 
	\geq
	\lim_{N \rightarrow \infty} D_N $,
	then
	$ \delta_1 - \delta_2 \in ( - D_N +1, D_N -1)$
	so the first equation in
	(\ref{eq:ltr:3.1}) gives that
	$ \alpha_1 = \alpha_2 $
	and
	$ \delta_1 = \delta_2 $.
	Moreover, the second equation from
	(\ref{eq:ltr:3.1}) gives that 
	$ \gamma_1 $ 
	is uniquely determined by        
	$ (\beta_1, \beta_2, \gamma_2)$,
	and the third equation gives that
	$( \gamma_2, \delta_2) $
	is uniquely determined by
	$ ( \alpha_2, \beta_2)$.
	Therefore
	\begin{align}
	\# \big\{ (i, j, l) \in [ M_N ]^3:\ &
	\Gamma( b_N, d_N ) (i, j) 
	= 
	\Ltr( B_N, D_N) (l, j)
	\big\} \label{eq:ltr:4.1}\\
	& \leq \#\big\{ (\alpha_1, \beta_1, \beta_2) : \
	\alpha_1 \in [ b_N ], \beta_1, \beta_2 \in [ d_N ]
	\big\} \nonumber \\
	& = b_N \cdot d_N^2 
	= \frac{M_N^2}{b_N}. \nonumber
	\end{align}
	On the other hand,  the second equation from
	(\ref{eq:ltr:3.1}) also gives that  
	$ \beta_1 $
	is uniquely determined by
	$( \beta_2, \gamma_1, \gamma_2)$
	while the third equation gives that
	$ (\alpha_2, \beta_2)$
	is uniquely determined by
	$ ( \gamma_2, \delta_2) $.
	Therefore 
	\begin{align}
	\# \big\{ (i, j, l) \in [ M_N ]^3:\ &
	\Gamma( b_N, d_N ) (i, j) 
	= 
	\Ltr( B_N, D_N) (l, j)
	\big\} \label{eq:ltr:4.2}\\
	& \leq \#\big\{ (\gamma_1, \gamma_2, \delta_2) : \
	\gamma_1, \gamma_2 \in [ B_N ], 
	\delta_2 \in [ D_N ]
	\big\} \nonumber \\
	& = B_N^2 \cdot D_N
	 = \frac{M_N^2}{D_N}. \nonumber
	\end{align}  
	Since
	$ \displaystyle
	\lim_{N \rightarrow \infty} b_N \cdot D_N = \infty $,
	equations (\ref{eq:ltr:4.1}) and (\ref{eq:ltr:4.2}) imply that the condition 
	($\mathfrak{c}.1$)
	 is satisfied.

	If
	$ \displaystyle
	\lim_{N \rightarrow \infty}
	D_N \geq 
	\lim_{N \rightarrow \infty}
	d_N  
	$, 
	then
	$ \beta_1 - \beta_2 \in
	( - d_N + 1, d_N -1 ) $ 
	so the second equation from 
	(\ref{eq:ltr:3.1})
	 gives that
	$ \beta_1 = \beta_2 $ 
	and
	$ \gamma_1 = \gamma_2 $. 
	Similar to the argument above, the first equation from
	(\ref{eq:ltr:3.1})
	 gives that
	$ \delta_1 $
	 is uniquely determined by
	$ (\alpha_1, \alpha_2, \delta_2)$,
	and the third equation gives that
	$(\alpha_2, \beta_2) $ 
	uniquely determines
	$ (\gamma_2, \delta_2) $. 
	Therefore
	\begin{align}
	\# \big\{ (i, j, l) \in [ M_N ]^3:\ &
	\Gamma( b_N, d_N ) (i, j) 
	= 
	\Ltr( B_N, D_N) (l, j)
	\big\} \label{eq:ltr:4.3}\\
	& \leq \#\big\{ (\alpha_1, \alpha_2, \beta_1) : \
	\alpha_1, \alpha_2 \in [ b_N ], \beta_1 \in [ d_N ]
	\big\} \nonumber \\
	& = b_N^2 \cdot d_N = 
	\frac{M_N^2}{ d_N }. \nonumber
	\end{align}  
	On the other hand, the first equation from
	(\ref{eq:ltr:3.1}) also gives that
	$(\gamma_1, \gamma_2, \beta_2) $
	uniquely determines 
	$ \beta_1 $, and, again, from the third equation 
	$(\alpha_2, \beta_2) $
	are uniquely determined by
	$ (\gamma_2, \delta_2) $. 
	Therefore
	\begin{align}
	\# \big\{ (i, j, l) \in [ M_N ]^3:\ &
	\Gamma( b_N, d_N ) (i, j) 
	= 
	\Ltr( B_N, D_N) (l, j)
	\big\} \label{eq:ltr:4.4}\\
	& \leq \#\big\{ (\gamma_1, \delta_1, \delta_2) : \
	\gamma_1 \in [ B_N ], \delta_1, \delta_2 \in [ D_N ]
	\big\} \nonumber \\
	& = B_N \cdot D_N^2 = 
	\frac{M_N^2}{ B_N }. \nonumber
	\end{align}
	Using now that
	$ \displaystyle
	\lim_{N \rightarrow \infty} B_N \cdot d_N
	= \infty $, 
	equations (\ref{eq:ltr:4.3})
	and (\ref{eq:ltr:4.4})
	imply that the condition from 
	($\mathfrak{c}.1$)
	 is satisfied.
	
	For the condition 
	($\mathfrak{c}.3$), 
	 we have that
	the equality
	\[ 
	\pi_1 \circ \Gamma(b_N, d_N)(i, j) = 
	\pi_1 \circ \Ltr (B_N, D_N) (l, j)
	\]
	is equivalent to  
	\begin{equation}\label{eq:ltr:5.1}
	\left\{
	\begin{array}{l}
	( \alpha_1 - 1) d_N + \beta_2
	=
	( \gamma_2 -1) D_N + \delta_1
	\\
	( \alpha_2 -1) d_N + \beta_2 
	= 
	( \gamma_2 -1) D_N + \delta_2.
	\end{array}
	\right.
	\end{equation}  

	Note that the triple
	$ (i, j, l) \in [ M_N ]^3 $ 
	is then uniquely determined by the
	5-tuple
	$ ( \alpha_1, \alpha_2, \beta_1, \beta_2, 
	\gamma_1)  $
	since
	first equation of (\ref{eq:ltr:5.1}) gives that
	$ ( \alpha_1, \beta_2, \gamma_2) $
	uniquely determines 
	$ \delta_1$
	and the second equation gives that
	$(\alpha_2, \beta_2) $ 
	uniquely determines
	$(\gamma_2, \delta_2) $.
	Therefore
	\begin{align}
	\# \big\{ (i, j, l)  \in & [ M_N ]^3
	: \
	\pi_1 \circ \Gamma(b_N, d_N ) (i, j) 
	=
	\pi_1 \circ \Ltr (B_N, D_N) ( l, j) 
	\big\} \label{eq:ltr:5.2}\\
	\leq & 
	\# \big\{
	(\alpha_1, \alpha_2, \beta_1, \beta_2, \gamma_1):\ 
	\alpha_1, \alpha_2 \in [ b_N ], 
	\beta_1, \beta_2 \in [ d_N ],
	\gamma_1 \in [ B_N ]
	\big\}\nonumber\\
	= & 
	b_N^2 \cdot d_N^2 \cdot B_N 
	= \frac{M_N^3}{D_N}.
	\nonumber
	\end{align}

	On the other hand, the triple 
	$(i, j, l) $
	is also uniquely determined by
	the 5-tuple
	$ (\beta_1, \gamma_1, \gamma_2, \delta_1, \delta_2) $
	since the first equation of
	(\ref{eq:ltr:5.1}) gives that 
	$( \beta_2, \gamma_2, \delta_1)$
	uniquely determines
	$ \alpha_1$
	and the second equation gives that
	$(\gamma_2, \delta_2)$
	uniquely determines
	$(\alpha_2, \beta_2)$. 
	Therefore
	\begin{align}
	\# \big\{ (i, j, l)  \in & [ M_N ]^3
	: \
	\pi_1 \circ \Gamma(b_N, d_N ) (i, j) 
	=
	\pi_1 \circ \Ltr (B_N, D_N) ( l, j) 
	\big\} \label{eq:ltr:5.3}\\
	\leq & 
	\# \big\{
	(\beta_1, \gamma_1, \gamma_2, \delta_1, \delta_2):\ 
	\beta_1 \in [ d_N ],
	\gamma_1, \gamma_2 \in [ B_N ],
	\delta_1, \delta_2 \in [ D_N ]
	\big\}\nonumber\\
	= & 
	d_N \cdot B_N^2 \cdot D_N^2 
	= \frac{M_N^3}{b_N}.
	\nonumber
	\end{align}
	So condition
	($\mathfrak{c}.3$)
	follows from relations
	(\ref{eq:ltr:5.2}), (\ref{eq:ltr:5.3})
	and  from 
	$ \displaystyle 
	\lim_{N \rightarrow \infty} b_N \cdot D_N = \infty $.   
	
	Finally, to show that
	$ \Gamma( b_N, d_N ) $
	and 
	$ \Ltr( B_N, D_N) $
	satisfy  condition 
	($\mathfrak{c}.2$),
we use that the equality
	\[
	\pi_2 \circ \Gamma (b_N, d_N ) (i, j) 
	= \pi_2 \circ \Ltr (B_N, D_N )(l, j)
	\]
	is equivalent to
	\begin{equation}\label{eq:ltr:6.1}
	\left\{
	\begin{array}{l}
	( \alpha_2 - 1) d_N + \beta_1
	=
	( \gamma_1 -1) D_N + \delta_2
	\\
	( \alpha_2 -1) d_N + \beta_2 
	= 
	( \gamma_2 -1) D_N + \delta_2.
	\end{array}
	\right.
	\end{equation}  
	In this case, the first equation of
	(\ref{eq:ltr:6.1}) gives that the couples
	$ ( \alpha_2, \beta_1) $ 
	and
	$ ( \gamma_1, \delta_2 )$ 
	uniquely determine each other; 
	the second equation gives that so do the couples
	$( \alpha_2, \beta_2) $
	and
	$( \gamma_2, \delta_2) $,
	hence the triple
	$ (i, j, l)$
	is now uniquely determined by either of the 5-tuples
	$( \alpha_1, \alpha_2, \beta_1, \beta_2, \delta_1)$
	and
	$ ( \alpha_1, \gamma_1, \gamma_2, \delta_1, \delta_2)$.
	Therefore
	\begin{align}
	\# \big\{
	(i, j, l)  \in & [ M_N ]^3
	: \
	\pi_2 \circ \Gamma(b_N, d_N ) (i, j) 
	=
	\pi_2 \circ \Ltr (B_N, D_N) ( l, j) 
	\big\} \label{eq:ltr:6.2}\\
	\leq & 
	\# \big\{
	(\alpha_1, \alpha_2, \beta_1, \beta_2, \delta_1) :\ 
	\alpha_1, \alpha_2 \in [ b_N ],
	\beta_1, \beta_2 \in [ d_N ],
	\delta_1 \in [ D_N ]
	\big\}\nonumber\\
	= & 
	b_N^2 \cdot d_N^2 \cdot D_N
	= \frac{M_N^3}{B_N},
	\nonumber
	\end{align}  
	respectively   
	\begin{align}
	\# \big\{
	(i, j, l)  \in & [M_N ]^3
	: \
	\pi_2 \circ \Gamma(b_N, d_N ) (i, j) 
	=
	\pi_2 \circ \Ltr (B_N, D_N) ( l, j) 
	\big\} \label{eq:ltr:6.3}\\
	\leq & 
	\# \big\{
	(\alpha_1, \gamma_1, \gamma_2, \delta_1, \delta_2) :\ 
	\alpha_1 \in [ b_N ],
	\gamma_1, \gamma_2 \in [ d_N ],
	\delta_1, \delta_2 \in [ D_N ]
	\big\}\nonumber\\
	= & 
	b_N \cdot B_N^2 \cdot D_N^2
	= \frac{M_N^3}{d_N},
	\nonumber
	\end{align} 
	So condition
	($\mathfrak{c}.2$)
	relations
	(\ref{eq:ltr:6.2}), (\ref{eq:ltr:6.3})
	and  from the condition
	$ \displaystyle 
	\lim_{N \rightarrow \infty} d_N \cdot B_N = \infty $.

	If property (2) holds true, then (3) follows from Lemma \ref{lemma:nf:1}.
	
	Finally, to show that property (3) implies (1),  suppose that
	$ \displaystyle 
	\lim_{N \rightarrow \infty } d_N \cdot B_N < \infty 
	 $  (the case 
	 $ \displaystyle
	 \lim_{N \rightarrow \infty} D_N \cdot b_N < \infty $ will follow by taking transposes).
	 
Let
$ \{ e_N \}_N $
be a sequence of positive integers such that
\[
\frac{1}{2} D_{ N} \leq d_{ N} \cdot e_N 
\leq  D_{ N}.
\] 	 
In particular, 
$ \displaystyle \lim_{N \rightarrow \infty }
e_N = \infty $,
and
\[
(d_{N} \cdot e_N)^2 \geq 
\frac{M_N^2}{4 B_{ N}^2}.
\]	 
Consider the set
\[
F = \big\{ ( \alpha_1 \cdot d_{ N} + \beta ,
\alpha_2 \cdot d_{N} + \beta ) :\ 
\alpha_1, \alpha_2 = 0, 1, \dots, e_N -1 , \beta \in [ d_{ N} ] 
\big\}.
\] 
Note that, if 
$ (i, j) \in F $
then
$ \Gamma(b_{ N}, d_{ N})(i, j) = 
\Ltr ( B_{ N}, D_{ N})(i, j) = (i,j) $,
hence
\[
\ \# \big\{ (i, j) \in [ M_N ]^2 :\ 
\Gamma( b_N, d_N ) (i, j) = \Ltr ( B_N, D_N ) (i, j) 
\big\}
\geq \# F, 
\]
but
\[
\#F =  d_{ N} \cdot e_N^2 \geq
\frac{M_N^2}{ 4 B_{ N}^2 d_{ N} }
\]
and the conclusion follows.	 
	\end{proof}

An immediate consequence of Theorem \ref{thm:ltr:main} is the following result.
\begin{cor}
	For any two sequences 
	$ \{ b_N \}_N $
	and 
	$ \{ d_N\}_N $
	such that
	$ b_N \cdot d_N = M_N $
	and the limits
	$ \displaystyle
	\lim_{N \rightarrow \infty} 
	b_{ N} $,
	and 
	$ \displaystyle
	\lim_{N \rightarrow \infty} 
	d_{ N} $ ,
	exist,
	we have that
	$W_N^{ \Gamma( b_N , d_N) } $
	and
	$ W_N^{ \ltr(b_N, d_N )} $
	are almost surely asymptotically free.
\end{cor}

Moreover, in the proof of Theorem \ref{thm:ltr:main} it is in fact shown that properties (1) and (3) are equivalent to the permutations
$ \Gamma( b_N, d_N ) $
and
$ \Ltr(B_N, D_N)$
satisfying the conditions
($\mathfrak{c}.1$),
($\mathfrak{c}.2$),
($\mathfrak{c}.3$). 
Hence we have the following.

\begin{cor}\label{cor:4:15}
	Let
	$\{ M_N \}_N $
	be a strictly increasing sequence of positive integers.
	Suppose that
	for each 
	$ N$,
	$W_N $ 
	is a
	 $M_N \times M_N $
	 Wishart matrix.
	
Suppose also that for each
$ i =1, 2, \dots, m + n $
there exist two sequences
$ \{ b_{i, N}\}_N $ 
and 
$ \{ d_{i, N}\}_N $
such that:
\begin{itemize}
	\item[(i)] $ b_{i, N} \cdot d_{i, N} = M_N $
	\item[(ii)] the limits 
	$ \displaystyle
	\lim_{N \rightarrow \infty} 
	b_{i, N} $,
	and 
	$ \displaystyle
	\lim_{N \rightarrow \infty} 
	d_{i, N} $ 
	exist.
	\item[(iii)] $ d_{1, N} \leq d_{2, N} \leq \dots 
	\leq d_{m, M} $
\end{itemize}
For 
$ i \leq m $,
denote 
$ \sigma_{i, N } = \Gamma(b_{i, N}, d_{i, N})$
and for
$ m + 1 \leq i \leq m + n $,
denote
$ \sigma_{i, N} = \Ltr(b_{i-m, N}, d_{i-m, N } ) $.
The following statements are then equivalent:
\begin{itemize}
	\item[(1)] The family 
	$ \big\{ W_N^{\sigma_{i, N} } :\ i = 1, 2, \dots, m + n \big\} $
	is almost surely asymptotically free
	\item[(2)] If
	$ k < l $,
	then 
	$\{ d_{k, N}\}_N $
		and
	$\{ d_{l, N} \}_N $	
	satisfy 
	condition 
	$(\mathfrak{w}.1) $
	if either
	$ l \leq m $ 
	or 
	$ m < k $,
	respectively satisfy the property
\[	
	\lim_{N \rightarrow \infty} d_{k, N} \cdot b_{l, N} = 
	\lim_{N \rightarrow \infty } b_{k, N} \cdot d_{l, N} 
	= \infty 
\]
	if $ k \leq m < l $.
	\item[(3)] Whenever
	$ a \neq b $, 
	we have that
	 \[
	\lim_{N \rightarrow \infty} \frac{1}{M_N^2}
	 \# \big\{  (i, j) :\ 
	  \sigma_{a, N}(i, j) = \sigma_{b, N}(i, j) \big\} =0. 
	\]
\end{itemize}
\end{cor}

\section{Second order fluctuations and almost sure convergence} 
\label{sec:fluctuations}

In this section we show that the covariances
\[
\mathrm{Cov}\big(\Tr( W_N^{ \sigma_{1, N}}
   W_N^{ \sigma_{2, N}} \cdots
    W_N^{ \sigma_{m, N}}  ),
     \Tr(W_N^{ \sigma_{ m+1, N}} \cdots
      W_N^{ \sigma_{ m+r, N}})   \big)
\]      
are bounded independently of $N$. This will be used to prove Theorem \ref{thm:52}, which shows that
convergence in moments will imply almost the sure convergence that we claimed in Section \ref{main:section}.

\begin{lemma}\label{lemma:cov}
 Suppose that 
 $ m, r $ 
 are two positive integers and that, for every  
 $ s \in [ m + r ]$,
 $\{  \sigma_{s, N} \}_N $ 
 is a sequence of symmetric permutations with
 each
 $ \sigma_{s, N}  $  
 being an element of
 $ \mathcal{S}( [ M_N]^2 ) $.
 Then there exist some positive 
 $ C ( m , r )$ 
 such that for each
  $ N $ 
  we have that
  \begin{align*}
   \mathrm{Cov}\big(\Tr( W_N^{ \sigma_{1, N}}
   W_N^{ \sigma_{2, N}} \cdots
    W_N^{ \sigma_{m, N}}  ),
     \Tr(W_N^{ \sigma_{ m+1, N}} \cdots
      W_N^{ \sigma_{ m+r, N}})   \big) < C(m, r).
  \end{align*}
\end{lemma}
\begin{proof}
 For 
 $ a $ and $ b $ 
 two positive integers, we shall denote by
  $ \cP_2^{a, b}(a+b, 2) $
  the set of all pair partitions 
  $ \pi $ 
  on 
  $ [ a + b ] $
  such that for 
  $ k  $ and $ \pi ( k ) $
   have different parities for each 
    $ k \in [ a + b ] $
  and there exists some 
  $ s \leq a $ 
  such that
  $ \pi (s ) > a $. 
  In other words, 
  \begin{align*}
  \cP_2^{ a, b } (b + b, 2 ) 
  = \{ 
  \pi \in \cP_2 ( a + b , 2 ) :\
  \pi \neq \pi_1 \oplus \pi_2 
  \textrm{ for any } &
  \pi_1 \in \cP_2 ( a, 2 ) \\
  &  \textrm{ and any } \pi_2 \in \cP_2 ( b, 2 )
   \}.
  \end{align*}
  (Here by $ \oplus $ we shall understand the concatenation of two multi-indices or of two pairings).
  
  Since, with the notations from Section \ref{section:vV}, we have 
  \begin{align*}
  E \big(\Tr( W_N^{ \sigma_{1, N}}
     W_N^{ \sigma_{2, N}} \cdots
      W_N^{ \sigma_{m, N}}  )
      \cdot
       \Tr & (W_N^{ \sigma_{ m+1, N}} \cdots
        W_N^{ \sigma_{ m+r, N}})   \big)\\
        & 
          = 
         \sum_{ \pi \in \cP_2 (2m + 2r, 2 ) }
         \sum_{ \substack{
          \vec{ u_1} \in \I ( m ) \\ 
          \vec{ u_2 } \in \I ( r ) } }
          v ( \pi,  \os,
           \vec{ u_1 } \oplus  \vec{ u_2 } ),
  \end{align*}
it follows that
 \begin{align*}
   \mathrm{Cov} \big(\Tr ( W_N^{ \sigma_{1, N}}  &
    W_N^{ \sigma_{2, N}} \cdots 
     W_N^{ \sigma_{m, N}}  ), 
      \Tr(W_N^{ \sigma_{ m+1, N}} \cdots
       W_N^{ \sigma_{ m+r, N}})   \big) \\
         = &
       \sum_{ \pi \in \cP_2^{ 2m, 2r} ( 2m + 2r, 2 )}
       \ \sum_{ \substack{
                 \vec{ u_1} \in \I ( m ) \\ 
                 \vec{ u_2 } \in \I ( r ) } }
                 v ( \pi,  \os,
                  \vec{ u_1 } \oplus  \vec{ u_2 } )\\
                   = &
   M_N^{ - ( m + r )}   
 \sum_{ \pi \in \cP_2^{ 2m, 2r} ( 2m + 2r, 2 )}
\# \mathcal{A}^{ m, r} ( \pi, \os ),
                \end{align*}
 where
 \begin{align*}
 \mathcal{A}^{ m, r} ( \pi, \os ) = \{ \vec{ u_1} \oplus \vec { u_2 }:\ \vec { u_1 } \in \I ( m ),
 \vec { u_2 } \in \I (r) 
 \textrm{ and }
 v( \pi, \os, \vec{ u_1} \oplus \vec{ u_2} )
 \neq 0  \}
 \end{align*}
  So it suffices to show that for any
  $ \os $ 
  and any
  $ \pi \in \cP_2^{ 2m, 2r} ( 2m + 2r, 2 ) $,           
 \begin{align}\label{ineq:34}
 \# \mathcal{A}^{ m, r} ( \pi, \os ) \leq \big( \max\{ M_N, P_N\} \big)^{ m + r }.
 \end{align}
 
 To show (\ref{ineq:34}), fix 
 $ \pi \in \cP_2^{ 2m, 2r} ( 2m + 2r, 2 )$. 
 Since
 $ 2m $ 
 is even and 
 $ p $ 
 connects only elements of different parities, there exists some
 even
 $ s < 2m $ 
 such that 
 $ \pi(s) > 2m $.
 Therefore, without loss of generality, via circular permutations of the sets 
 $ [ 2m ] $ 
 and
 $ [ 2 m + 2r ] \setminus [ 2m ] $,
 it suffices to show
 (\ref{ineq:34})
 for pairings $ \pi $ 
 such that
 $ \pi( 2m ) = 2 m + 1 $.

 For each
 $ k \in [ m + r ] $,
  denote
 \begin{align*}
 B_{\os} ( k ) =  \# & \{
  (i_s, j_s, j_{-s}, i_{-s})_{ 1 \leq s \leq k }:\
  \textrm {there exists some } 
   ( i_l, j_l , i_{ -l}, j_{ -l})_{k < l \leq m + r}\\
   & \textrm{ such that } 
   ( i_1, j_1, i_{ -1}, j_{-1}, \dots, i_{m+r}, j_{ m  + r },
   j_{ - m -r}, i_{ - m - r } )
   \in  \mathcal{A}^{ m, r} ( \pi, \os ) 
   \},
 \end{align*}
and remark that, 
\begin{align}\label{b:k}
 B_{ \os} ( k ) \leq B_{ \os } ( k -1)
  \cdot
   \big( \max\{ M_N, P_N \} \big)^{ 2 -
    \#\big( \pi ( \{ 2k -1, 2k \} ) \cap [ 2k ] \big) }.
\end{align}

To prove (\ref{b:k}), we fix 
$ ( i_s, j_s, j_{ - s }, i_{ -s })_{ s \leq  k -1 } $ 
and, using a similar argument to the proof of Lemma \ref{lemma:1:1}, we shall show that the number of tuples
$ ( i_k, j_k, i_{ - k }, j_{ - k } ) $
such that
$ ( i_s, j_s, j_{ - s }, i_{ -s })_{ s \leq  k } $
can be completed to an element from
$ \mathcal{A}^{ m, r} ( \pi, \os ) $
  is at most
$  \big( \max\{ M_N, P_N \} \big)^{ 2 -
    \#\big( \pi ( \{ 2k -1, 2k \} ) \cap [ 2k ] \big) } $.
    
    Since
     $ i_k = i_{ - ( k - 1 ) } $
     and
     $ j_k = j_{ - k } $,
     we have that
     $( i_k, j_k, i_{ - k }, i_{ - k } ) $
     is uniquely determined by
     $ i_{ - ( k - 1 ) } $
     and by
     $ ( j_k, i_{ - k } ) $, 
     so
  $ B_{ \os } ( k ) \leq B_{ \os } ( k -1) \cdot M_N P_N  $. 
  
  If 
  $ \pi ( 2k -1 ) , \pi ( 2k ) \leq 2k - 2 $, 
  then besides the conditions above, we have that
  $ l_k = l_{  - \pi (2k -1) } $, $ j_k = j_{ \pi ( 2k -1 ) } $
  and
  $ l_{ - k } = l_{ \pi (2k ) } $.
  So
  $ ( l_k, j_k, l_{ -k } ) $,
  hence
  $ ( i_k, j_k, j_{ - k},  i_{ - k})$
  is uniquely determined by
  $ ( i_s, j_s, j_{ -s }, i_{ - s } )_{ s \leq k -1 } $.
  Then
   $ B_{ \os} (k) = B_{ \os } ( k -1 ) $.
   
   If 
     $ \pi ( 2k -1 ) = 2k $, 
   then
    $ l_k = l_{ - k } $. 
    Since 
 $ \sigma_{ k, N} $
  is symmetric, we obtain that
  $ i_k = i_{ - k } $, 
 so
 $ ( i_k, j_{ k } , j_{ - k }, i_{ - k } ) $
 is uniquely determined by
  $ j_k $ 
  and by
  $ ( i_s, j_s, j_{ -s }, i_{ - s } )_{ s \leq k -1 } $.
  Therefore
   $ B_{ \os} ( k) \leq B_{ \os} ( k -1) \cdot P_N $.
   
   If 
   $ \pi ( 2k -1 ) = q \leq 2k -2 $
   and
   $ \pi ( 2k ) > 2k $, 
   then
   $ l_{k } = l_{q} $ 
   and
   $ j_k = j_q $,
   so
   $ ( i_k, j_k, j_{ - k},  i_{ - k})$ 
  is uniquely determined by
   $ l_{ - k } $ 
   and by
   $ ( i_s, j_s, j_{ -s }, i_{ - s } )_{ s \leq k -1 } $;
  hence
  $ B_{ \os}( k) \leq B_{ \os} ( k -1 ) \cdot M_N $.
  The case
   $ \pi( 2k -1 ) > 2k $ and $ \pi (2k ) \leq 2k - 2 $
   is similar, so the proof for (\ref{b:k}) is complete. 
   
   Clearly
   \begin{align*}
   B_{ \os} ( 1) \leq
      \#\{ ( i_1, j_1, j_{ -1}, i_{ -1} ) : j_1 = j_{ - 1}\} =  M_N^2 P_N
   \end{align*}  
   so applying (\ref{b:k}) gives
   \begin{align}
   B_{ \os} (m-1) \leq ( M_N + P_N )^{2m -1 -
    \#\{ s \leq 2 ( m -1) :\ \pi (s) \leq 2 ( m -1) \} } .
   \end{align}
   
   On the other hand, note that
   \begin{align}\label{b:k:m}
   B_{ \os } ( m ) \leq B_{\os}( m-1) \cdot (M_N + P_N )^{ 1 -
    \# ( \pi ( \{ 2m-1, 2m \} ) \cap [ 2m ] )   }
   \end{align}
   so (\ref{b:k}) gives
   \begin{align*}
    B_{ \os } ( m ) \leq ( M_N + P_N )^{2m  -
        \#\{ s \leq 2 m :\ \pi (s) \leq 2 m \} } . 
   \end{align*}
    
   To show (\ref{b:k:m}), we use that 
   $ i_{ - m -1} = i_m $
   and
   $ i_{ -m } = i_1 $.
   So if both
   $ \pi ( 2 m -1 ) $
   and
   $ \pi ( 2 m )  $
   are greater than 
   $ 2m $, then
    $ ( i_m, j_m, j_{ -m }, i_{ - m } ) $
    is uniquely determined by 
     $ j_m $ 
     and 
     $ ( i_s, j_s, j_{ -s }, i_{ - s } )_{ s \leq k -1 } $.
     Hence, in this case, 
    $ B_{\os} ( m ) \leq B_{ \os} ( m -1 ) \cdot P_N $.
    If 
     $ \pi ( 2m -1) = q  \leq 2m -2 $
     then we also have that
     $ j_m = j_q $
     therefore
     $ ( i_m, j_m, j_{ -m }, i_{ - m } ) $
         is uniquely determined by 
          $ ( i_s, j_s, j_{ -s }, i_{ - s } )_{ s \leq k -1 } $,
          and
  $ B_{\os} ( m ) \leq B_{ \os} ( m -1 ) $.     
  
  Next, notice that 
  \begin{align}\label{b:k:m1}
  B_{ \os } ( m + 1 ) \leq ( M_N + P_N )^{ 2( m + 1) -
   \# \{ s \leq 2( m + 1 ): \pi (s) \leq 2( m +1 ) \} }
  \end{align}
  
  To prove the inequality above it suffices to show that (\ref{b:k}) holds true also for 
  $ k = m + 1 $. 
  To show it, note first that, since we are under the assumption 
  $ \pi ( 2m ) = 2m + 1 $,
  we have 
  $ l_{ - m} = l_{ m + 1 } $ 
  and 
  $ j_{ - m } = j_{ m + 1 } $.
  Hence, if
  $ \pi ( 2m + 2 ) > 2m +2 $, 
  the tuple 
  $ ( i_m, j_m, j_{ - m }, i_{ - m } ) $
  is uniquely determined by 
  $ l_{ -( m + 1 )} $
  and by
  $ ( i_s, j_s, j_{ -s }, i_{ - s } )_{ s \leq m } $, 
  so
  $ B_{ \os } ( m  + 1 ) \leq B_{ \os } (m ) \cdot M_N $.
  If
   $ \pi ( 2m + 2) \leq 2m $
   then we also have
    $ l_{ - ( m + 1) } = l_{ \pi ( 2m + 2) } $.
    In this case
  $ ( i_m, j_m, j_{ - m }, i_{ - m } ) $
  is uniquely determined by
   $ ( i_s, j_s, j_{ -s }, i_{ - s } )_{ s \leq m } $, 
     so
     $ B_{ \os } ( m  + 1 ) \leq B_{ \os } (m ) $,
     which completes the proof of (\ref{b:k:m1}).
  
    Finally, inequality (\ref{b:k:m1}) and another application of 
    (\ref{b:k}) give
    \begin{align*}
 B_{ \os} ( m + r ) \leq   ( M_N + P_N )^{ 2 m + 2 r - 
 \# \{ s \in [ 2m + 2 r ]: \pi (s) \leq 2m + 2 r \} } 
 =
 ( M_N + P_N )^{ m + r } . 
    \end{align*}
   But
    $ B_{ \os} ( m + r ) = \# \mathcal{A}^{m, r} ( \pi, \os ) $,
    so the proof is complete.
\end{proof}

 Assume now that for each positive integer 
 $ N $, the entries of the 
 $ M_N \times P_N $ 
 Ginibre matrix 
 $ G_N $  are independent identically distributed Gaussian random variables of variance 
 $ \displaystyle \frac{1}{ \sqrt{ M_N} } $
  from the same probability space 
 $ (\Omega, P ) $. 
 In particular, for each
 $ \omega \in \Omega $ 
 and each positive integer $ N $,
 the product 
 $ W_N ( \omega ) = G_N ( \omega ) G_N ( \omega)^\ast $ 
 is a
 $ M_N \times M_N $ 
 positive complex matrix.

 Standard techniques in probability (see, for example,
 \cite{billingsley}, or Chapter 4 of \cite{mingo-speicher}) give the following consequence of Lemma \ref{lemma:cov}.
 
 \begin{thm}\label{thm:52}
  If for each
  $ s \in [ m ] $
  and
  $ N \in \mathbb{N} $
  $ \sigma_{ s, N } $
  is a permutation from
  $ \mathcal{S}( [ M_N ]^2 ) $
  and
   $ L $ 
   is a complex number such that
   \begin{align*}
  \lim_{N \rightarrow \infty}
  E \circ \tr \big( 
  W_N^{ \sigma_{ 1, N}} W_N^{ \sigma_{2, N}}
   \cdots W_N^{ \sigma_{ m, N}}    
    \big) = L 
   \end{align*}
 then, almost surely on 
 $ ( \Omega, P) $,
 we have that
 \begin{align*}
 \lim_{N \rightarrow \infty} 
 \tr \big( 
   W_N(\omega)^{ \sigma_{ 1, N}} 
    W_N(\omega)^{ \sigma_{2, N}} 
      \cdots
       W_N(\omega)^{ \sigma_{ m, N} } 
 \big) = L.
 \end{align*}  
 \end{thm}

\begin{proof}
 For each positive integer 
 $ N $,
 consider
  $ f_N : \Omega \rightarrow \mathbb{C} $
  given by
 \begin{align*}
 f_N (\omega) =
 \tr \big( 
    W_N(\omega)^{ \sigma_{ 1, N}}
     W_N(\omega)^{ \sigma_{2, N}} 
       \cdots
        W_N(\omega)^{ \sigma_{ m, N} } 
  \big).
 \end{align*}
For Lemma \ref{lemma:cov}, there is some positive 
$ C $
such that
$ \displaystyle \textrm{Var}( f_N ) < \frac{C}{N^2} $ 
so, for any
$ \varepsilon > 0 $,
 Chebyshev's inequality gives
 \begin{align*}
 P \big(  \{ \omega \in \Omega :\ | f_N ( \omega ) - E ( f_N ) |
 \geq \varepsilon 
   \} \big)
  \leq
  \frac{C}{\varepsilon^2 N^2 },
 \end{align*}
 hence
 \begin{align*}
 \sum_{ N \geq 1 }
 P \big( 
 \{ \omega \in \Omega :\ | f_N ( \omega ) - E ( f_N ) |
  \geq \varepsilon 
    \}
 \big)
 < 
 \infty
 \end{align*}
 and the almost sure convergence for 
 $ f_N $
 follows from the Borel-Cantelli Lemma.
\end{proof}




\bibliographystyle{alpha}

\begin{thebibliography}{10}

\bibitem{arizmendi}
O.~Arizmendi, I.~Nechita, and C.~Vargas, 
On the asymptotic distribution of block-modified random matrices, 
\textit{J. Math. Phys.} 57 (2016), no. 1, 015216, 25~pp. 

\bibitem{aubrun} G.~Aubrun, Partial Transposition Of Random
  States And Non-Centered Semicircular Distributions,
  \textit{Random Matrix Theory and Applications}, \textbf{1}
  (2012), no. 2, 1250001, 29 pp.
  
\bibitem{aubrun2} G. Aubrun, S. Szarek,  E. Werner, 
Hastings's additivity counterexample via Dvoretzky's theorem,
\textit{Commun. Math. Phys.} 305(1), 85--97 (2011).
	
\bibitem{banica-nechita} T. Banica and I. Nechita, 
Asymptotic eigenvalue distributions of block-transposed Wishart matrices, \textit{J. Theor. Probab.} February 2012, 1--15 (2012).
	
\bibitem{billingsley}
P. Billingsley \emph{Probability and Measure Theory}	 (3rd ed.), Wiley, New York, 1995.
	
\bibitem{edelman} A. Edelman, Y. Wang,
Random Matrix Theory and Its Innovative Applications, Fields Institute Communications: \textit{Advances in Applied Mathematics, Modeling, and Computational Science} (2013): 91--116.

\bibitem{fs} M.~Fukuda and P.~\'Sniady, Partial Transpose Of Random Quantum States: Exact Formulas And Meanders, \textit{J. Math. Phys.} 54 (2013), no. 4, 042202, 23 pp.

\bibitem{horod} M. Horodecki, P. Horodecki, and R. Horodecki, Separability of mixed states: Necessary and sufficient conditions, \textit{Phys. Lett. A} 223(1–2), 1--8 (1996). 


\bibitem{jason}
S. Janson, \emph{Gaussian Hilbert Spaces}, Cambridge Tracts in Mathematics, vol. 129, Cambridge University Press, Cambridge, 1997.



\bibitem{mingo-popa-transpose} J.A. Mingo, M. Popa, Freeness and the transposes of unitarily invariant random matrices, \textit{J.~of Funct. Anal.} 271(4), 2014, 883--921.

\bibitem{mingo-popa-wishart}
J.A. Mingo, M. Popa, Freeness and the partial transpose of Wishart random matrices,  \textit{Canad. J. Math.}, 71 (2019), 659--681.

\bibitem{mingo-speicher} 
J.A. Mingo, R. Speicher, \emph{Free Probability and Random Matrices}. Fields Institute Monographs, Vol. 35, Springer, New York, 2017.

\bibitem{muirhead} R.J. Muirhead, \emph{Aspects of multivariate statistical theory}, Wiley Series in Probability and Statistics Vol. 197, Hoboken, NJ 2009.

\bibitem{nica-speicher}
A. Nica, R. Speicher, \emph{Lectures on the Combinatorics of Free Probability},
London Mathematical Society Lecture Note Series, vol. 335,
Cambridge University Press, 2006.

\bibitem{popa-guassian}
M.Popa, \emph{Asymptotic free independence and entry permutations for Gaussian random matrices}, arXiv:1812.01692v2.

\bibitem{wishart}
J. Wishart,  {The generalized product moment distribution in samples from a normal multivariate population}. \textit{Biometrika} 20A, (1928), 32--52. 



\end{thebibliography}


\end{document}